\pdfoutput=1
\documentclass[12pt]{article}
\usepackage{amsmath,amssymb,amsfonts,here}
\usepackage{tikz}
\begin{document}
\newtheorem{theorem}{Theorem}[section]
\newtheorem{corollary}[theorem]{Corollary}
\newtheorem{lemma}[theorem]{Lemma}
\newtheorem{remark}[theorem]{Remark}
\newtheorem{example}[theorem]{Example}
\newtheorem{proposition}[theorem]{Proposition}
\newtheorem{definition}[theorem]{Definition}
\newtheorem{assumption}[theorem]{Assumption}
\def\emptyset{\varnothing}
\def\setminus{\smallsetminus}
\def\id{{\mathrm{id}}}
\def\G{{\mathcal{G}}}
\def\E{{\mathcal{E}}}
\def\H{{\mathcal{H}}}
\def\C{{\mathbb{C}}}
\def\N{{\mathbb{N}}}
\def\Q{{\mathbb{Q}}}
\def\R{{\mathbb{R}}}
\def\Z{{\mathbb{Z}}}
\def\Path{{\mathrm{Path}}}
\def\Str{{\mathrm{Str}}}
\def\st{{\mathrm{st}}}
\def\tr{{\mathrm{tr}}}
\def\opp{{\mathrm{opp}}}
\def\a{{\alpha}}
\def\be{{\beta}}
\def\de{{\delta}}
\def\e{{\varepsilon}}
\def\si{{\sigma}}
\def\la{{\lambda}}
\def\th{{\theta}}
\def\lan{{\langle}}
\def\ran{{\rangle}}
\def\isom{{\cong}}
\newcommand{\Hom}{\mathop{\mathrm{Hom}}\nolimits}
\newcommand{\End}{\mathop{\mathrm{End}}\nolimits}
\def\qed{{\unskip\nobreak\hfil\penalty50
\hskip2em\hbox{}\nobreak\hfil$\square$
\parfillskip=0pt \finalhyphendemerits=0\par}\medskip}
\def\proof{\trivlist \item[\hskip \labelsep{\bf Proof.\ }]}
\def\endproof{\null\hfill\qed\endtrivlist\noindent}

\title{$\alpha$-induction for bi-unitary connections}
\author{
{\sc Yasuyuki Kawahigashi}\\
{\small Graduate School of Mathematical Sciences}\\
{\small The University of Tokyo, Komaba, Tokyo, 153-8914, Japan}\\
{\small e-mail: {\tt yasuyuki@ms.u-tokyo.ac.jp}}
\\[0,40cm]
{\small Kavli IPMU (WPI), the University of Tokyo}\\
{\small 5--1--5 Kashiwanoha, Kashiwa, 277-8583, Japan}
\\[0,05cm]
{\small and}
\\[0,05cm]
{\small iTHEMS Research Group, RIKEN}\\
{\small 2-1 Hirosawa, Wako, Saitama 351-0198,Japan}}
\maketitle{}
\centerline{\sl Dedicated to the memory of Vaughan Jones}

\begin{abstract}
The tensor functor called 
$\alpha$-induction produces a new unitary fusion category from a
Frobenius algebra object, or a $Q$-system, in a braided
unitary fusion category.  In the operator algebraic 
language, it gives extensions of endomorphism of $N$
to $M$ arising from a subfactor
$N\subset M$ of finite index and finite depth which gives a
braided fusion category of endomorphisms of $N$.
It is also understood in terms of 
Ocneanu's graphical calculus.  We study this 
$\alpha$-induction for bi-unitary
connections, which give a characterization of finite-dimensional
nondegenerate commuting squares and gives certain $4$-tensors
appearing in recent studies of
$2$-dimensional topological order.  We show that the resulting
$\alpha$-induced bi-unitary connections are flat if we
start with a commutative Frobenius algebra, or a local $Q$-system.
Examples related
to chiral conformal field theory and the Dynkin diagrams are
presented.
\end{abstract}

\section{Introduction}

A fusion category \cite{EGNO} has recently emerged as a new
type of symmetry in a wide range of topics in
mathematics and physics.  Theory of operator
algebras gives a nice framework to study this type of new
symmetries, as exemplified by discovery of the Jones
polynomial for knots \cite{J2} from the Jones theory of
subfactors \cite{J}.  We have three operator algebraic
realizations of a fusion category based on
endomorphisms of a type III factor \cite{L1}, \cite{L2},
bimodules over type II$_1$ factors \cite[Chapter 9]{EK2},
and bi-unitary connections \cite[Section 3]{AH}.
The last approach based on bi-unitary connections recently
has renewed interest because of its relations to 2-dimensional
statistical physics \cite{BMWSHV}, \cite{K4}.

A certain $4$-tensor, a (finite) family of complex numbers indexed
by four indices, is studied in physics literature such as 
\cite{BMWSHV}, \cite{LFHSV}.   This has been identified with a 
bi-unitary connection in the subfactor sense in
\cite{K4}, and further studies \cite{K5}, \cite{K6}, \cite{K7}
have followed in this direction.  Also see \cite{H} for a recent
development.  Bi-unitary connections and related topics
are also recently studied in \cite{CEM}, \cite{DGGJ}.
This approach to a fusion category based on bi-unitary connections
has advantage that everything is described with finite dimensional
matrices and thus in principle computable on computers, while
endomorphisms of a type III factor and bimodules over type II$_1$
factors are infinite dimensional.  This computability is one reason
physicists are interested in this approach recently.  Our
aim in this paper to study \textit{$\alpha$-induction} in this framework
of bi-unitary connections.

The tensor functor $\alpha$-induction originates in the operator
algebraic studies of \textit{chiral conformal field theory} \cite{LR}.
In this context, our basic object is a \textit{conformal net},
which is a family of von Neumann algebras
parametrized by intervals contained in the circle $S^1$.
It has a representation theory of
\textit{superselection sectors} and this has a structure of
\textit{braiding} \cite{FRS1}, \cite{FRS2}.  
If a conformal net has a certain finiteness
property called \textit{complete rationality}, we obtain a
\textit{modular tensor category} of its representations \cite{KLM}.
See a short review \cite{K3a} and a longer review \cite{K3}
for a general theory of this topic.

In a classical representation theory, if we have a representation
of a subgroup $H\subset G$, we have a method of \textit{induction}
to obtain a representation of $G$.  We have a similar, but more
subtle, induction machinery for representation theory of
conformal nets, depending on braiding structure.
Suppose we have an inclusion $A\subset B$
of completely rational conformal nets and $\lambda$ is
a representation of $A$ given as a Doplicher-Haag-Roberts
endomorphism of $A(I)$ for some fixed interval $I$ in $S^1$.
Then we can extend this endomorphism to $B(I)$ using the
braiding.  This was first defined by Longo-Rehren \cite{LR}
and further studied by Xu \cite{X1} and B\"ockenhauer-Evans
\cite{BE1}, \cite{BE2}, \cite{BE3}, where this was named
as \textit{$\alpha$-induction}.  

In a completely different setting,
Ocneanu had a machinery of \textit{flat connections} to study
subfactors \cite{O1}, \cite{O2}, which consists of entries of a
large unitary matrix.  Then using a braiding structure of
flat connections on the Dynkin diagrams of type $A$, he introduced
a graphical calculus to construct new fusion categories related 
to the Goodman-de la Harpe subfactors \cite{GHJ} in
\cite{O3}.  We have unified the two theories of $\alpha$-induction
and Ocneanu's graphical calculus in a fully general setting
of braided fusion categories
and proved various basic properties such as appearance of
modular invariants in \cite{BEK1}, \cite{BEK2} in the framework
of endomorpshisms of type III factors.  In this setting,
$\alpha$-induction is understood as a method of extending
an endomorphism of a factor $N$ to another factor $M\supset N$
using braiding, where we do not assume anything about conformal
field theory.  This $\alpha$-induction has been also understood
in the language of bimodules and more abstract braided fusion
category.

Our motivation for this paper is as follows.  First, it is
nice to have a concrete realization of $\alpha$-induction in 
the setting of bi-unitary connections due to its
finite dimensionality and relations to statistical physics.
Secondly, it gives a
generalization of Ocneanu's graphical calculus in a much more
general setting.  We also
note that the same mathematical structure as $\alpha$-induction 
appears in the context of anyon condensation \cite{BS}, which 
gives another reason to study $\alpha$-induction in this
setting.  (See \cite[Table 1]{K} for this direction.)

Our main result is Theorem \ref{equiv} which shows our bi-unitary
connection defined in Definition \ref{W4} is a correct description
of $\alpha$-induction.
We have a remark that though we use 
an endomorphism framework \cite{BEK1}, \cite{BEK2}, 
we only need abstract properties of $\alpha$-induction, so we can
also use a bimodule framework or a setting based on
more abstract fusion categories for $\alpha$-induction for
our results.

This work was partially supported by 
JST CREST program JPMJCR18T6 and
Grants-in-Aid for Scientific Research 19H00640
and 19K21832. A part of this work was done at Microsoft Station Q
at Santa Barbara and University of California, Berkeley.  
The author thanks for their hospitality.

\section{Preliminaries on braided fusion categories, subfactors
and $\alpha$-induction}
\label{prelim}

We give our setting for $\alpha$-induction for
a subfactor $N\subset M$ and endomorphisms of $N$
as in \cite{BEK1}.  Let $N$ be a type III factor.  Throughout
this paper, we assume the following for 
$\Delta$ which is a finite set of mutually inequivalent
irreducible endomorphisms of $N$ of finite dimension.
(See \cite[Definition 2.1]{BEK1}.)

\begin{assumption}\label{assum0}{\rm
We have the following for $\Delta$.

(1) The identity automorphism is in $\Delta$.

(2) For any $\lambda\in\Delta$, we have another element
$\mu\in\Delta$ which is equivalent to $\bar\lambda$.

(3) For any $\lambda,\mu\in\Delta$, the composition
$\lambda\mu$ decomposes into a direct sum of irreducible
endomorphisms each of which is equivalent to one in $\Delta$.

(4) The set $\Delta$ has a \textit{braiding} 
$\e(\lambda,\mu)\in\Hom(\lambda\mu,\mu\lambda)$ as in
\cite[Definition 2.2]{BEK1}.
}\end{assumption}

We further make an assumption on connectedness of certain
bipartite graphs as follows, which will be necessary in the
following Section.  Note that this assumption depends
on a choice of $\mu$, an irreducible endomorphism of $N$ in
$\Delta$.  This is assumed throughout the paper, except
for Section \ref{exa}.

\begin{assumption}\label{assump}{\rm
Consider a bipartite graph defined as follows.
Let both even and odd vertex sets be labeled with
the elements in $\Delta$.  The number of edges between
the even vertex $\nu_1$ and the odd vertex $\nu_2$ is
given by $\dim\Hom(\nu_1\mu,\nu_2)$.  We assume that this
graph is connected.
}\end{assumption}

Endomorphisms of $N$ with finite dimension
whose irreducible decompositions
give endomorpshisms equivalent to ones in $\Delta$ produce
a fusion category, where objects are such endomorphisms and
morphisms are intertwiners between endomorphisms.
(See \cite[Section 2.1]{BEK1} for more details on intertwiners.
Also see \cite{BKLR} for a recent treatment of fusion categories
and subfactors.  See \cite{EGNO} for a more abstract
and algebraic treatment of fusion categories.)
We consider a subfactor $N\subset M$ with finite index
whose canonical
endomorphism $\theta$ decomposes into a sum of endomorphisms
equivalent to ones in $\Delta$.  Such a subfactor automatically
has a finite depth.  For a fixed $N$ and
a fusion category of such endomorphisms, an extension $M$ of $N$
is in a bijective correspondence to a \textit{$Q$-system}
as in \cite[Section 6]{L}.  In algebraic literature, this is
often called a \textit{Frobenius algebra}.  (In this paper, we
only consider $C^*$-tensor categories. Also see \cite[Theorem 2.3]{M}
for a bimodule formulation of a $Q$-system.)

Since we assume to have a braiding, our fusion
category is a braided fusion category.  If a braiding
is nondegenerate in the sense of \cite[Definition 2.3]{BEK1},
we say that this braided fusion category is a \textit{modular
tensor category}, but we do \textit{not} assume this
nondegeneracy in this paper.  We have a notion of \textit{locality},
$\e(\theta,\theta)\gamma(v)=\gamma(v)$, for a $Q$-system corresponding
to a subfactor $N\subset M$.  (Here $\gamma$ is the canonical
endomorphism of $M$ and an isometry $v\in M$ satisfies
$vx=\gamma(x)v$ for $x\in M$ and $M=Nv$.)  The name locality
comes from locality of an extension of a conformal net
in the setting of \cite[Theorem 4.9]{LR} and 
it was called \textit{chiral locality} in \cite{BEK1}.
A local $Q$-system is also called a \textit{commutative
Frobenius algebra} in algebraic literature.
We deal with \textit{both} local and non-local $Q$-systems 
in this paper.

The procedure called
\textit{$\alpha$-induction} was defined in \cite[Proposition 3.9]{LR} 
as follows.
\[
\alpha_\lambda^\pm=\bar\iota^{-1}\cdot
\mathrm{Ad}(\varepsilon^\pm(\lambda,\theta))\cdot\lambda\cdot\bar\iota,
\]
where $\lambda$ is an endomorphism in $\Delta$, $\iota$ is
the inclusion map $N\hookrightarrow M$, $\theta=\bar\iota\cdot\iota$
is the dual canonical endomorphism of $N\subset M$, and $\pm$
stands for a choice of positive/negative braiding.
It is a nontrivial fact that
$\mathrm{Ad}(\varepsilon^\pm(\lambda,\theta))\cdot
\lambda\cdot\bar\iota(x)$ is in the image of $\bar\iota$ 
for $x\in M$.  We have $\alpha_\lambda^\pm(x)=\lambda(x)$ for
$x\in N$ and $\alpha_\lambda^\pm(v)=
\varepsilon^\pm(\lambda,\theta)^*v$.
See \cite[Section 3]{BEK1} for basic properties of $\alpha$-induction
in this setting.

In a very different setting,
Ocneanu used Fig.~\ref{chigen} to represent a
chiral generator in the double triangle algebra.
(Also see \cite[Fig.~47]{BEK1}.)
It was identified with the $\alpha$-induction in
\cite[Theorem 5.3]{BEK1}.  (See \cite[Section 4]{BEK1} for the
double triangle algebra and
\cite[Fig.~4.1]{BEK1} for a graphical convention
involving small half circles.)

\begin{figure}[H]
\begin{center}
\begin{tikzpicture}[scale=0.8]
\draw [thick](1,1)--(5,1);
\draw [thick](1,3)--(5,3);
\draw (3,1)--(3,1.8);
\draw (3,2.2)--(3,3);
\draw (2,3) arc (180:360:1);
\draw (1.8,3) arc (180:360:0.2);
\draw (2.8,3) arc (180:360:0.2);
\draw (3.2,1) arc (0:180:0.2);
\draw (3.8,3) arc (180:360:0.2);
\draw (0.4,1.9)node{$\displaystyle\sum_{a,b,\mu}$};
\draw (3.5,1.8)node{$\lambda$};
\draw (3,1.4)node[left]{$\mu$};
\draw (1,1)node[below]{$b$};
\draw (5,1)node[below]{$b$};
\draw (1,3)node[above]{$a$};
\draw (2.5,3)node[above]{$b$};
\draw (3.5,3)node[above]{$b$};
\draw (5,3)node[above]{$a$};
\end{tikzpicture}
\caption{The chiral generator $p_\lambda^+$}
\label{chigen}
\end{center}
\end{figure}


For various diagrams, we use the convention in \cite[Section 3]{BEK1}.
In particular, our convention is as in Fig.\ref{triple}, where
$T$ is an isometry in $\Hom(\lambda,\mu\nu)$ as in
\cite[Fig.~21]{BEK1}.
Note that rotation invariance of this type of diagrams
is due to the \textit{Frobenius reciprocity} for endomorphisms
due to Izumi \cite{I1}, \cite{I2}. 
Also as in \cite[Section 4]{BEK1}, we draw a thin wire for
an $N$-$N$ morphism, a thick wire for an $N$-$M$ or $M$-$N$ 
morphism, and a very thick wire for an $M$-$M$ morphism.
For a thick wire, we can tell whether it stands for an
$N$-$M$ or $M$-$N$ morphism from a diagram.

\begin{figure}[H]
\begin{center}
\begin{tikzpicture}[scale=0.8]
\draw [-to](2,3)--(2,2);
\draw [-to](2,2)--(1,1);
\draw [-to](2,2)--(3,1);
\draw [-to](8,3)--(9,2);
\draw [-to](10,3)--(9,2);
\draw [-to](9,2)--(9,1);
\draw (1,1)node[below]{$\mu$};
\draw (3,1)node[below]{$\nu$};
\draw (9,1)node[below]{$\lambda$};
\draw (2,3)node[above]{$\lambda$};
\draw (8,3)node[above]{$\mu$};
\draw (10,3)node[above]{$\nu$};
\draw (2,1.8)node[below]{$T$};
\draw (9,2.2)node[above]{$T^*$};
\draw (4.5,2)node
{$=\sqrt[4]{\displaystyle\frac{d_\mu d_\nu}{d_\lambda}}T,$};
\draw (11.5,2)node
{$=\sqrt[4]{\displaystyle\frac{d_\mu d_\nu}{d_\lambda}}T^*.$};
\end{tikzpicture}
\caption{Graphical convention for normalization of an intertwiner}
\label{triple}
\end{center}
\end{figure}

\section{$\alpha$-induction for bi-unitary connections}
\label{alpha}

We now produce a family of new bi-unitary connections
from the setting as in the previous Section.
We simply say a connection for a bi-unitary 
connection in this paper.

We first recall the definition of a connection.
(See \cite[Section 11.3]{EK2}, \cite[Section 3]{AH},
\cite[Section 2]{K5}.  There is an issue of connectedness
of certain graphs and we follow conventions of 
\cite[Section 3]{AH} on this matter.)

We have four finite bipartite graphs $\G, \G', \H, \H'$.
The vertex set $V_0$ is common for even vertices of
$\G$ and $\H$.  Similarly, the vertex sets $V_1,V_2,V_3$
are common for odd vertices of $\H$ and $\G'$, 
even vertices of $\G'$ and $\H'$, and odd vertices
of $\G$ and $\H'$, respectively.  
The four graphs satisfy some properties about 
the Perron-Frobenius
eigenvalues and eigenvectors as in \cite[Section 2]{K5}. 
We choose edge $\xi_0,\xi_1,\xi_2,\xi_3$ from the
graphs $\H, \G', \H', \G$, respectively so that they
make a closed square called a \textit{cell}.  A map $W$ called
a \textit{connection} assigns a complex number to each of such cells.
This complex value is represented with a diagram in Fig.~\ref{conn0}.
This map $W$ satisfies axioms called \textit{bi-unitarity} as
in \cite[Definition 2.2]{K5}.  (See also Fig.~\ref{conn14},
\ref{conn15}, \ref{conn16} below.)
The name ``bi-unitarirty'' means
that each number in Fig.~\ref{conn0} is an entry of a unitary
matrix in two ways, one after normalization arising from the
Perron-Frobenius eigenvector entries.
Now we require that the graphs
$\G$ and $\G'$ to be connected, but we do \textit{not} require
this for $\H$ and $\H'$.  (This convention is different from
the one in \cite{K5}, and the same as in \cite{AH}.)
We have an equivalence relation for connections on the
same four graphs as in Remark after \cite[Theorem 3]{AH}.
We call the graphs $\G,\G',\H,\H'$ the horizontal top graph,
the horizontal bottom graph, the vertical left graph and
the vertical right graph of $W$, respectively.  We also
call the vertices in $V_0,V_1,V_2,V_3$ the upper left vertices,
the lower left vertices, the lower right vertices and
the upper right vertices, respectively.

\begin{figure}[H]
\begin{center}
\begin{tikzpicture}
\draw [-to](1,1)--(2,1);
\draw [-to](1,2)--(2,2);
\draw [-to](1,2)--(1,1);
\draw [-to](2,2)--(2,1);
\draw (2,1.5)node[right]{$\xi_2$};
\draw (1,1.5)node[left]{$\xi_0$};
\draw (1.5,2)node[above]{$\xi_3$};
\draw (1.5,1)node[below]{$\xi_1$};
\draw (1,1)node[below left]{$x_1$};
\draw (2,1)node[below right]{$x_2$};
\draw (1,2)node[above left]{$x_0$};
\draw (2,2)node[above right]{$x_3$};
\draw (1.5,1.5)node{$W$};
\end{tikzpicture}
\caption{The four graphs for a connection}
\label{conn0}
\end{center}
\end{figure}

It is well-known that this bi-unitarity condition is characterized 
in terms of a \textit{commuting square} as we now explain below.
A commuting square of finite dimensional $C^*$-algebras
$\begin{array}{ccc}
A&\subset& B\\
\cap && \cap\\
C&\subset&D
\end{array}$
with a trace on $D$
is characterized by the following mutually equivalent conditions.
(See \cite[Proposition 9.51]{EK2}, for example.)

\begin{enumerate}
\item The conditional expectation $E_B$ from $D$
to $B$ with respect to trace restricted to $C$ coincides
with the conditional expectation $E_A$ from $C$
to $A$ with respect to trace.
\item The conditional expectation $E_C$ from $D$
to $C$ with respect to trace restricted to $B$ coincides
with the conditional expectation $E_A$ from $B$
to $A$ with respect to trace.
\end{enumerate}

This notion was originally considered in \cite{P1}.
We say that this commuting square is \textit{nondegenerate}
if the span $BC$ is equal to $D$.  (Being nondegenerate is
also sometimes said to be \textit{symmetric}.)
A nondegenerate square of finite dimensional $C^*$-algebras is
described with a connection, for which we do not know yet whether
it satisfies bi-unitarity, as in \cite[Section 11.2]{EK2}.
Then it has been proved in \cite[Theorem 1.10]{Sc} that the square gives
a commuting square if and only the connection satisfies
bi-unitarity.  (Also see \cite[Theorem 11.2]{EK2}.)  
In this sense, having a bi-unitary connection and having a
nondegenerate commuting square of finite dimensional $C^*$-algebras
are the same thing.

A commuting square
also naturally appears from a subfactor $N\subset M$ with
finite Jones index as
$\begin{array}{ccc}
M'\cap M_k&\subset& M'\cap M_{k+1}\\
\cap && \cap\\
N'\cap M_k&\subset& N'\cap M_{k+1}
\end{array}$, where $N\subset M \subset M_1\subset M_2\subset\cdots$
is the Jones tower arising from the Jones basic construction
\cite{J}.  (Also see \cite[Section 9.6]{EK2} for such a
commuting square.)
If the original subfactor $N\subset M$ is hyperfinite, of type II$_1$
and of \textit{finite depth}, then the above commuting square
recovers the original subfactor $N\subset M$ by Popa's theorem
\cite{P2}.  In this sense, a connection encodes complete
information about such a subfactor.  A connection corresponding
to such a commuting square arising from a subfactor $N\subset M$
satisfies a special extra property called \textit{flatness},
which was introduced by Ocneanu \cite{O1}, \cite{O2} and studied
in \cite{K1}.
(See \cite[Section 11.4]{EK2} for more discussions on flatness.)

We first define the connection $W_1(\lambda,\mu)$, but we make
a remark on one issue again.  In this paper, we use operator
algebraic realization
of a braided fusion category and $\alpha$-induction based
on endomorpshisms of a type III factor as in \cite{BEK1} and
\cite{BEK2}, such as \cite[Fig.~30]{BEK1}, but this is simply because
this was the first place where all necessary details were
worked out, and we can equally use other formulations based on
bimodules or abstract fusion categories.  Such a choice of
formulation does not cause any change in our results here.

\begin{definition}\label{W1}{\rm
Let $\nu_1,\nu_2,\nu_3,\nu_4\in \Delta$.
Consider the diagram in Fig.~\ref{conn2}.
By composing isometries $T_1\in\Hom(\nu_1\mu,\nu_2)$,
$T_2\in\Hom(\lambda\nu_2,\nu_4)$,
$T_3\in\Hom(\lambda\nu_1,\nu_3)$, and
$T_4\in\Hom(\nu_3\mu,\nu_4)$, we obtain a complex number
$T_4 T_3 \lambda(T_1^*)T_2^*\in\Hom(\nu_4,\nu_4)$.
We define the connection $W_1(\lambda,\mu)$ by this number
and draw Fig.~\ref{conn1}.
}\end{definition}

We note that the left vertical bipartite
graph in this Definition is defined as follows. 
One set of vertices is
given by $\Delta$ and the other is the same.  The number
of edges between the vertices $\nu_1$ and $\nu_3$ is given
by $\dim\Hom(\lambda\nu_1,\nu_3)$.  The other graphs are defined
similarly, and this remark applies to all the Definitions of
the connections in this Section.

\begin{figure}[H]
\begin{center}
\begin{tikzpicture}[scale=2.5]
\draw [-to](1.2,1)--(1.8,1);
\draw [-to](1.2,2)--(1.8,2);
\draw [-to](1,1.8)--(1,1.2);
\draw [-to](2,1.8)--(2,1.2);
\draw (1,1)node{$\nu_3\mu$};
\draw (2,1)node{$\nu_4$};
\draw (1,2)node{$\lambda\nu_1\mu$};
\draw (2,2)node{$\lambda\nu_2$};
\draw (1.5,2)node[above]{$\lambda(T_1)$};
\draw (1.5,1)node[below]{$T_4$};
\draw (1,1.5)node[left]{$T_3$};
\draw (2,1.5)node[right]{$T_2$};
\end{tikzpicture}
\caption{The diagram for the connection $W_1(\lambda,\mu)$
in Fig.~\ref{conn1}}
\label{conn2}
\end{center}
\end{figure}

\begin{figure}[H]
\begin{center}
\begin{tikzpicture}[scale=1.7]
\draw [-to](1,1)--(2,1);
\draw [-to](1,2)--(2,2);
\draw [-to](1,2)--(1,1);
\draw [-to](2,2)--(2,1);
\draw (1,1)node[below left]{$\nu_3$};
\draw (2,1)node[below right]{$\nu_4$};
\draw (1,2)node[above left]{$\nu_1$};
\draw (2,2)node[above right]{$\nu_2$};
\draw (1.5,1)node[below]{$T_4$};
\draw (1.5,2)node[above]{$T_1$};
\draw (1,1.5)node[left]{$T_3$};
\draw (2,1.5)node[right]{$T_2$};
\draw (1.5,1.5)node{$W_1(\lambda,\mu)$};
\end{tikzpicture}
\caption{The standard diagram for a connection $W_1(\lambda,\mu)$}
\label{conn1}
\end{center}
\end{figure}

We often drop labels for the connection and/or intertwiners
as long as no confusion arises, and simply
draw a diagram in Fig.~\ref{conn1a}.

\begin{figure}[H]
\begin{center}
\begin{tikzpicture}
\draw [-to](1,1)--(2,1);
\draw [-to](1,2)--(2,2);
\draw [-to](1,2)--(1,1);
\draw [-to](2,2)--(2,1);
\draw (1,1)node[below left]{$\nu_3$};
\draw (2,1)node[below right]{$\nu_4$};
\draw (1,2)node[above left]{$\nu_1$};
\draw (2,2)node[above right]{$\nu_2$};
\end{tikzpicture}
\caption{The standard diagram for a connection $W_1(\lambda,\mu)$}
\label{conn1a}
\end{center}
\end{figure}

\begin{figure}[H]
\begin{center}
\begin{tikzpicture}[scale=0.6]
\draw (6.5,8.5) arc (0:180:2);
\draw (4,7) arc (0:180:1.5);
\draw (5,6) arc (0:180:1);
\draw (1,5) arc (180:360:1);
\draw (2,4) arc (180:360:1.5);
\draw (3.5,2.5) arc (180:360:1.5);
\draw (6.5,8.5)--(6.5,2.5);
\draw (1,5)--(1,7);
\draw (3,5)--(3,6);
\draw (5,4)--(5,6);
\draw (2.5,2.5)node{$\nu_3$};
\draw (3.8,8.3)node{$\nu_2$};
\draw (1,5.5)node[left]{$\lambda$};
\draw (3,5.5)node[right]{$\nu_1$};
\draw (5,5.5)node[right]{$\mu$};
\draw (6.5,5.5)node[right]{$\bar\nu_4$};
\end{tikzpicture}
\caption{The diagram for the connection $W_1(\lambda,\mu)$
in Fig.~\ref{conn1}}
\label{conn3}
\end{center}
\end{figure}

We draw a diagram as in Fig.~\ref{conn3} which represents
a complex number in the standard convention as in
\cite[Section 3]{BEK1}.  (Note that in Fig.~\ref{conn3}, 
we drop orientations of wires which go
from the top to the bottom.)  This complex number 
is equal to the one represented by the composition of
isometries as in Fig.~\ref{conn2} multiplied by
$\sqrt{d_\lambda d_\mu d_{\nu_1} d_{\nu_4}}$
by the standard convention in Fig.~\ref{triple}.

Then unitarity of the connection $W_1(\lambda,\mu)$ is
represented as in Fig.~\ref{conn14}, \ref{conn15}.
Here the bars above the right squares denote the complex
conjugates.

\begin{figure}[H]
\begin{center}
\begin{tikzpicture}
\draw [-to](2,1)--(3,1);
\draw [-to](2,2)--(3,2);
\draw [-to](2,2)--(2,1);
\draw [-to](3,2)--(3,1);
\draw (2,1)node[below left]{$\nu_3$};
\draw (3,1)node[below right]{$\nu_4$};
\draw (2,2)node[above left]{$\nu_1$};
\draw (3,2)node[above right]{$\nu_2$};
\draw (2.5,1)node[below]{$T_4$};
\draw (2.5,2)node[above]{$T_1$};
\draw (2,1.5)node[left]{$T_3$};
\draw (3,1.5)node[right]{$T_2$};
\draw [-to](4.5,1)--(5.5,1);
\draw [-to](4.5,2)--(5.5,2);
\draw [-to](4.5,2)--(4.5,1);
\draw [-to](5.5,2)--(5.5,1);
\draw (4.5,1)node[below left]{$\nu_3$};
\draw (5.5,1)node[below right]{$\nu_4$};
\draw (4.5,2)node[above left]{$\nu_1$};
\draw (5.5,2)node[above right]{$\nu'_2$};
\draw (5,1)node[below]{$T_4$};
\draw (5,2)node[above]{$T'_1$};
\draw (4.5,1.5)node[left]{$T_3$};
\draw (5.5,1.5)node[right]{$T'_2$};
\draw [thick](4.3,2.6)--(5.7,2.6);
\draw (0.8,1.4)node{$\displaystyle\sum_{\nu_3,T_3,T_4}$};
\draw (7.5,1.5)node{$=\delta_{T_1,T'_1} \delta_{T_2,T'_2}$};
\end{tikzpicture}
\caption{Unitarity (1) for the connection $W_1(\lambda,\mu)$}
\label{conn14}
\end{center}
\end{figure}

\begin{figure}[H]
\begin{center}
\begin{tikzpicture}
\draw [-to](2,1)--(3,1);
\draw [-to](2,2)--(3,2);
\draw [-to](2,2)--(2,1);
\draw [-to](3,2)--(3,1);
\draw (2,1)node[below left]{$\nu_3$};
\draw (3,1)node[below right]{$\nu_4$};
\draw (2,2)node[above left]{$\nu_1$};
\draw (3,2)node[above right]{$\nu_2$};
\draw (2.5,1)node[below]{$T_4$};
\draw (2.5,2)node[above]{$T_1$};
\draw (2,1.5)node[left]{$T_3$};
\draw (3,1.5)node[right]{$T_2$};
\draw [-to](4.5,1)--(5.5,1);
\draw [-to](4.5,2)--(5.5,2);
\draw [-to](4.5,2)--(4.5,1);
\draw [-to](5.5,2)--(5.5,1);
\draw (4.5,1)node[below left]{$\nu'_3$};
\draw (5.5,1)node[below right]{$\nu_4$};
\draw (4.5,2)node[above left]{$\nu_1$};
\draw (5.5,2)node[above right]{$\nu_2$};
\draw (5,1)node[below]{$T'_4$};
\draw (5,2)node[above]{$T_1$};
\draw (4.5,1.5)node[left]{$T'_3$};
\draw (5.5,1.5)node[right]{$T_2$};
\draw [thick](4.3,2.6)--(5.7,2.6);
\draw (0.8,1.4)node{$\displaystyle\sum_{\nu_2,T_1,T_2}$};
\draw (7.5,1.5)node{$=\delta_{T_3,T'_3} \delta_{T_4,T'_4}$};
\end{tikzpicture}
\caption{Unitarity (2) for the connection $W_1(\lambda,\mu)$}
\label{conn15}
\end{center}
\end{figure}

Crossing symmetry for the connection $W_1(\lambda,\mu)$
is given as in Fig.~\ref{conn16}, where
$\tilde T_1$ and $\tilde T_4$ stand for the Frobenius
duals of $T_1$ and $T_4$, respectively.  This is a well-known 
relation arising from the tetrahedral symmetry of 
the $6j$-symbols as in \cite[Definition 12.15]{EK2}, but we include
a simple argument for this, because we need a similar argument
later for $W_3(\lambda)$ as in Fig.~\ref{conn8}.

\begin{figure}[H]
\begin{center}
\begin{tikzpicture}
\draw [-to](1,1)--(2,1);
\draw [-to](1,2)--(2,2);
\draw [-to](1,2)--(1,1);
\draw [-to](2,2)--(2,1);
\draw (1,1)node[below left]{$\nu_4$};
\draw (2,1)node[below right]{$\nu_3$};
\draw (1,2)node[above left]{$\nu_2$};
\draw (2,2)node[above right]{$\nu_1$};
\draw (1.5,1)node[below]{$\tilde T_4$};
\draw (1.5,2)node[above]{$\tilde T_1$};
\draw (1,1.5)node[left]{$T_2$};
\draw (2,1.5)node[right]{$T_3$};
\draw [-to](5.5,1)--(6.5,1);
\draw [-to](5.5,2)--(6.5,2);
\draw [-to](5.5,2)--(5.5,1);
\draw [-to](6.5,2)--(6.5,1);
\draw (5.5,1)node[below left]{$\nu_3$};
\draw (6.5,1)node[below right]{$\nu_4$};
\draw (5.5,2)node[above left]{$\nu_1$};
\draw (6.5,2)node[above right]{$\nu_2$};
\draw (6,1)node[below]{$T_4$};
\draw (6,2)node[above]{$T_1$};
\draw (5.5,1.5)node[left]{$T_3$};
\draw (6.5,1.5)node[right]{$T_2$};
\draw [thick](5.3,2.6)--(6.7,2.6);
\draw (3.75,1.5)node{$=\displaystyle
\sqrt{\frac{d_{\nu_1}d_{\nu_4}}{d_{\nu_2}d_{\nu_3}}}$};
\end{tikzpicture}
\caption{Crossing symmetry for the connection $W_1(\lambda,\mu)$}
\label{conn16}
\end{center}
\end{figure}

\begin{figure}[H]
\begin{center}
\begin{tikzpicture}[scale=0.6]
\draw (6.5,9) arc (0:180:1.5);
\draw (5,7.5) arc (0:180:1.5);
\draw (3,6.5) arc (0:180:1);
\draw (1,4.5) arc (180:360:1.5);
\draw (2.5,3) arc (180:360:2);
\draw (3,5.5) arc (180:360:1);
\draw (6.5,9)--(6.5,3);
\draw (1,6.5)--(1,4.5);
\draw (3,6.5)--(3,5.5);
\draw (5,7.5)--(5,5.5);
\draw (2,9)node{$\nu_3$};
\draw (4,3)node{$\nu_2$};
\draw (1,6)node[left]{$\lambda$};
\draw (3,6)node[left]{$\nu_1$};
\draw (5,6)node[right]{$\mu$};
\draw (6.5,6)node[right]{$\bar\nu_4$};
\end{tikzpicture}
\caption{A vertical reflection of Fig.~\ref{conn3}}
\label{conn3b}
\end{center}
\end{figure}

We first make  a vertical
reflection of the diagram in Fig.~\ref{conn3} to obtain
Fig.~\ref{conn3b}.
We then redraw Fig.~\ref{conn3b} to obtain Fig.~\ref{conn3c}, 
which represents the complex number given by the diagram
in Fig.~\ref{conn3d} multiplied by
$\sqrt{d_\lambda d_{\nu_2} d_\mu d_{\nu_3}}$.

\begin{figure}[H]
\begin{center}
\begin{tikzpicture}[scale=0.6]
\draw (6.5,8.5) arc (0:180:2);
\draw (4,7) arc (0:180:1.5);
\draw (5,6) arc (0:180:1);
\draw (1,5) arc (180:360:1);
\draw (2,4) arc (180:360:1.5);
\draw (3.5,2.5) arc (180:360:1.5);
\draw (6.5,8.5)--(6.5,2.5);
\draw (1,5)--(1,7);
\draw (3,5)--(3,6);
\draw (5,4)--(5,6);
\draw (2.5,2.5)node{$\nu_4$};
\draw (3.8,8.3)node{$\nu_1$};
\draw (1,5.5)node[left]{$\lambda$};
\draw (3,5.5)node[right]{$\nu_2$};
\draw (5,5.5)node[right]{$\bar\mu$};
\draw (6.5,5.5)node[right]{$\bar\nu_3$};
\end{tikzpicture}
\caption{Redrawing of Fig.~\ref{conn3b}}
\label{conn3c}
\end{center}
\end{figure}

\begin{figure}[H]
\begin{center}
\begin{tikzpicture}[scale=1.7]
\draw [-to](1,1)--(2,1);
\draw [-to](1,2)--(2,2);
\draw [-to](1,2)--(1,1);
\draw [-to](2,2)--(2,1);
\draw (1,1)node[below left]{$\nu_4$};
\draw (2,1)node[below right]{$\nu_3$};
\draw (1,2)node[above left]{$\nu_2$};
\draw (2,2)node[above right]{$\nu_1$};
\draw (1.5,1)node[below]{$\tilde T_4$};
\draw (1.5,2)node[above]{$\tilde T_1$};
\draw (1,1.5)node[left]{$T_2$};
\draw (2,1.5)node[right]{$T_1$};
\draw (1.5,1.5)node{$W_1(\lambda,\bar\mu)$};
\end{tikzpicture}
\caption{The diagram for a connection $W_1(\lambda,\bar\mu)$}
\label{conn3d}
\end{center}
\end{figure}

This shows we have crossing symmetry, Fig.~\ref{conn16}.
This, together with unitarity of the
left hand side of Fig.~\ref{conn16}, shows bi-unitarity of
the connection $W_1(\lambda,\mu)$.

We now prove that the fusion category of endomorphisms of
$N$ arising from $\Delta$ and the one arising from the
connections $W_1(\lambda,\mu)$ for various $\lambda$ and
fixed $\mu$ are equivalent.  This is where we need
Assumption \ref{assump}.
Then we have the following theorem.

\begin{theorem}
\label{equiv0}
Under Assumptions \ref{assum0} and \ref{assump},
the fusion category arising from the 
connections $W_1(\lambda,\mu)$
for all $\lambda\in\Delta$
is equivalent to the one arising from 
endomorphisms $\lambda\in\Delta$ of $N$.
\end{theorem}

\begin{proof}
By the description of the intertwiners between open string
bimodules in the proof of \cite[Theorem 3]{AH},
we have a natural injective linear map from 
$\Hom(\lambda_1\lambda_2,\lambda_3)$
to $\Hom(W_1(\lambda_1,\mu))W_1(\lambda_2,\mu)),W_1(\lambda_3,\mu))$
for $\lambda_1,\lambda_2,\lambda_3\in\Delta$.
The morphism space 
$\Hom(W_1(\lambda_1,\mu))W_1(\lambda_2,\mu)),W_1(\lambda_3,\mu))$
is described with the higher relative commutants of the corresponding
subfactor by the proof of \cite[Theorem 3]{AH}.  These higher relative
commutants are described with $C_{k,-1}$ for the
subfactor $M_0\subset M_1$ in the proof of
\cite[Theorem 3.3]{S}.  They are described with the intertwiner 
spaces of the $N$-$N$ morphisms as in the proof of
\cite[Theorem 3.3]{S}.  This means that the dimensions of the
two intertwiner spaces $\Hom(\lambda_1\lambda_2,\lambda_3)$
and $\Hom(W_1(\lambda_1,\mu))W_1(\lambda_2,\mu)),W_1(\lambda_3,\mu))$
are the same, and thus the above natural linear map is surjective.
The compositions of intertwiners in the two fusion categories
are also compatible with this identification, so
they are equivalent.
\end{proof}

Then we see that the structure of the \textit{braided} fusion
category of these endomorphisms of $N$ passes to that of
these connections.


We next introduce two types of connections.  The first is
the easier one $W_2(\mu)$.

\begin{definition}\label{W2}{\rm
Let $\nu_1,\nu_2\in\Delta$ and $a_1,a_2$ be irreducible
$M$-$N$ morphisms arising from $\Delta$ and the subfactor $N\subset M$.
Fig.~\ref{conn3g} represents the complex number
$T_4 T_3 \iota(T_1)^* T_2^*\in\Hom(a_2,a_2)$
for isometries $T_1\in\Hom(\nu_1\mu,\nu_2)$,
$T_2\in\Hom(\iota\nu_2,a_2)$,
$T_3\in\Hom(\iota\nu_1,a_1)$,
$T_4\in\Hom(a_1\nu_2,a_2)$, where
$a_1$ and $a_2$ are $M$-$N$ morphisms.  We define the
connection $W_2(\mu)$ by this number, and use
the diagram as in Fig.~\ref{conn4a} to represent this
connection.
}\end{definition}

Note that Remark after Definition \ref{W1} applies here again.
For example, the left vertical bipartite
graph in this Definition is defined as follows. 
One set of vertices is
given by $\Delta$ and the other is given by the representatives
of the irreducible $M$-$N$ morphisms.
The number of edges between the vertices $\nu_1\in\Delta$ 
and an $M$-$N$ morphism $a_1$ is given
by $\dim\Hom(\iota\nu_1,a_1)$. 

\begin{figure}[H]
\begin{center}
\begin{tikzpicture}[scale=2.5]
\draw [-to](1.2,1)--(1.8,1);
\draw [-to](1.2,2)--(1.8,2);
\draw [-to](2,1.8)--(2,1.2);
\draw [-to](1,1.8)--(1,1.2);
\draw (1,1)node{$a_1\mu$};
\draw (2,1)node{$a_2$};
\draw (1,2)node{$\iota\nu_1\mu$};
\draw (2,2)node{$\iota\nu_2$};
\draw (1.5,2)node[above]{$\iota(T_1)$};
\draw (1.5,1)node[below]{$T_4$};
\draw (1,1.5)node[left]{$T_3$};
\draw (2,1.5)node[right]{$T_2$};
\end{tikzpicture}
\caption{The diagram for the connection $W_2(\mu)$}
\label{conn3g}
\end{center}
\end{figure}

\begin{figure}[H]
\begin{center}
\begin{tikzpicture}
\draw [-to](1,1)--(2,1);
\draw [-to](1,2)--(2,2);
\draw [-to](1,2)--(1,1);
\draw [-to](2,2)--(2,1);
\draw (1,1)node[below left]{$a_1$};
\draw (2,1)node[below right]{$a_2$};
\draw (1,2)node[above left]{$\nu_1$};
\draw (2,2)node[above right]{$\nu_2$};
\end{tikzpicture}
\caption{The standard diagram for a connection $W_2(\mu)$}
\label{conn4a}
\end{center}
\end{figure}


We next introduce the other connection, $W_3(\lambda)$.

\begin{definition}\label{W3}{\rm
Let $\nu_1,\nu_2\in\Delta$ and $a_1,a_2$ be irreducible
$M$-$N$ morphisms arising from $\Delta$ and the 
subfactor $N\subset M$.
Fig.~\ref{conn3e} represents the complex number
$T_4 \alpha_\lambda^+(T_3) 
\E^+(\lambda,\iota)^* \iota(T_1)^* T_2^*\in\Hom(a_2,a_2)$
for isometries $T_1\in\Hom(\lambda\nu_1,\nu_2)$,
$T_2\in\Hom(\iota\nu_2,\nu_4)$,
$T_3\in\Hom(\iota\nu_1,a_1)$,
$T_4\in\Hom(\alpha_\lambda^+ a_1,a_2)$, which is represented with
the connection diagram in Fig.~\ref{conn4}.  We define the
connection $W_3(\lambda)$ by this number.
Here $\E^\pm$ is defined on page 455 below (14) in \cite{BEK1}
and we recall Fig.~\ref{conn5}, which is taken from
\cite[Fig.~30]{BEK1}.
}\end{definition}

\begin{figure}[H]
\begin{center}
\begin{tikzpicture}[scale=2.5]
\draw [-to](1.2,1)--(1.8,1);
\draw [-to](1.2,2)--(1.8,2);
\draw [-to](1,1.8)--(1,1.7);
\draw [-to](1,1.3)--(1,1.2);
\draw [-to](2,1.8)--(2,1.2);
\draw (1,1)node{$\alpha^+_\lambda a_1$};
\draw (1,1.5)node{$\alpha^+_\lambda\iota\nu_1$};
\draw (2,1)node{$a_2$};
\draw (1,2)node{$\iota\lambda\nu_1$};
\draw (2,2)node{$\iota\nu_2$};
\draw (1.5,2)node[above]{$\iota(T_1)$};
\draw (1.5,1)node[below]{$T_4$};
\draw (1,1.75)node[left]{$\E^+(\lambda,\iota)^*$};
\draw (1,1.25)node[left]{$\alpha_\lambda^+(T_3)$};
\draw (2,1.5)node[right]{$T_2$};
\end{tikzpicture}
\caption{The diagram for the connection $W_3(\lambda)$}
\label{conn3e}
\end{center}
\end{figure}

\begin{figure}[H]
\begin{center}
\begin{tikzpicture}
\draw [-to](1,1)--(2,1);
\draw [-to](1,2)--(2,2);
\draw [-to](1,2)--(1,1);
\draw [-to](2,2)--(2,1);
\draw (1,1)node[below left]{$a_1$};
\draw (2,1)node[below right]{$a_2$};
\draw (1,2)node[above left]{$\nu_1$};
\draw (2,2)node[above right]{$\nu_2$};
\end{tikzpicture}
\caption{The standard diagram for a connection $W_3(\lambda)$}
\label{conn4}
\end{center}
\end{figure}

\begin{figure}[H]
\begin{center}
\begin{tikzpicture}
\draw [thick,-to](1,2)--(1.4,1.6);
\draw [thick,-to](1.6,1.4)--(2,1);
\draw [-to](2,2)--(1.5,1.5);
\draw [ultra thick,-to](1.5,1.5)--(1,1);
\draw (1,1)node[below]{$\alpha_\lambda^+$};
\draw (2,1)node[below]{$\iota$};
\draw (1,2)node[above]{$\iota$};
\draw (2,2)node[above]{$\lambda$};
\end{tikzpicture}
\caption{The braiding operator $\E^+(\lambda,\iota)^*$}
\label{conn5}
\end{center}
\end{figure}

\begin{figure}[H]
\begin{center}
\begin{tikzpicture}[scale=0.6]
\draw [thick](5,11.5) arc (0:180:1.5);
\draw [thick](2,10.5) arc (90:180:1);
\draw (3,9.5) arc (0:90:1);
\draw (4,8.5) arc (0:180:1);
\draw (3,5.5) arc (270:360:1);
\draw [thick](2,6.5) arc (180:270:1);
\draw [thick](2,3.5) arc (270:360:1);
\draw [ultra thick](1,4.5) arc (180:270:1);
\draw [thick](2,2.5) arc (180:360:1.5);
\draw [thick](5,2.5)--(5,11.5);
\draw [thick](2,11.5)--(2,10.5);
\draw [thick](1,9.5)--(1,7.5);
\draw [thick](3,5.5)--(3,4.5);
\draw [thick](2,3.5)--(2,2.5);
\draw [ultra thick](1,6.5)--(1,4.5);
\draw (4,6.5)--(4,8.5);
\draw (2,7.5)--(2,8.5);
\draw [ultra thick](1,6.5)--(1.5,7);
\draw (1.5,7)--(2,7.5);
\draw [thick](1,7.5)--(1.4,7.1);
\draw [thick](1.6,6.9)--(2,6.5);
\draw (2,2.7)node[left]{$a_2$};
\draw (3,10.5)node{$\nu_2$};
\draw (1,8)node[left]{$\iota$};
\draw (2,8)node[right]{$\lambda$};
\draw (4,8)node[right]{$\nu_1$};
\draw (1,5)node[left]{$\alpha_\lambda^+$};
\draw (3,4.7)node[right]{$a_1$};
\end{tikzpicture}
\caption{The diagram for the connection $W_3(\lambda)$
in Fig.~\ref{conn4}}
\label{conn6}
\end{center}
\end{figure}

The complex number represented by the diagram in Fig.~\ref{conn6}
is equal to the one represented by the connection diagram in
Fig.~\ref{conn4} multiplied by
$\sqrt{d_\iota d_\lambda d_{\nu_1} d_{a_2}}$
by the standard convention in Fig.~\ref{triple}.
Note that we drop orientations of wires which go
from the top to the bottom in Fig.~\ref{conn6}.

We make a vertical reflection of Fig.~\ref{conn6} to get
Fig.~\ref{conn7}.  We again drop orientations of wires which
go from the top to the bottom.
Note that the complex number values given
by these two diagrams are complex conjugate to each other.

\begin{figure}[H]
\begin{center}
\begin{tikzpicture}[scale=0.6]
\draw [thick](5,11.5) arc (0:180:1.5);
\draw [ultra thick](2,10.5) arc (90:180:1);
\draw [thick](3,9.5) arc (0:90:1);
\draw [thick](3,8.5) arc (90:180:1);
\draw [thick](1,4.5) arc (180:270:1);
\draw (4,7.5) arc (0:90:1);
\draw (2,5.5) arc (180:360:1);
\draw (2,3.5) arc (270:360:1);
\draw [thick](2,2.5) arc (180:360:1.5);
\draw [thick](5,11.5)--(5,2.5);
\draw [ultra thick](1,7.5)--(1,9.5);
\draw [thick](1,4.5)--(1,6.5);
\draw [thick](2,10.5)--(2,11.5);
\draw [thick](2,2.5)--(2,3.5);
\draw (2,5.5)--(2,6.5);
\draw (4,5.5)--(4,7.5);
\draw [thick](3,8.5)--(3,9.5);
\draw [thick](1,6.5)--(1.4,6.9);
\draw [thick](1.6,7.1)--(2,7.5);
\draw [ultra thick](1,7.5)--(1.5,7);
\draw (1.5,7)--(2,6.5);
\draw (2,11)node[left]{$a_2$};
\draw (3,9)node[right]{$a_1$};
\draw (1,9)node[left]{$\alpha^+_\lambda$};
\draw (1,6)node[left]{$\iota$};
\draw (2,6)node[right]{$\lambda$};
\draw (4,6)node[right]{$\nu_1$};
\draw (3,3.5)node{$\nu_2$};
\end{tikzpicture}
\caption{A vertical reflection of Fig.~\ref{conn6}}
\label{conn7}
\end{center}
\end{figure}

By redrawing Fig.~\ref{conn7}, we have Fig.~\ref{conn8}.

\begin{figure}[H]
\begin{center}
\begin{tikzpicture}[scale=0.6]
\draw [thick](5,11.5) arc (0:180:1.5);
\draw [thick](2,10.5) arc (90:180:1);
\draw (3,9.5) arc (0:90:1);
\draw (4,8.5) arc (0:180:1);
\draw (3,5.5) arc (270:360:1);
\draw [thick](2,6.5) arc (180:270:1);
\draw [thick](2,3.5) arc (270:360:1);
\draw [ultra thick](1,4.5) arc (180:270:1);
\draw [thick](2,2.5) arc (180:360:1.5);
\draw [thick](5,2.5)--(5,11.5);
\draw [thick](2,11.5)--(2,10.5);
\draw [thick](1,9.5)--(1,7.5);
\draw [thick](3,5.5)--(3,4.5);
\draw [thick](2,3.5)--(2,2.5);
\draw [ultra thick](1,6.5)--(1,4.5);
\draw (4,6.5)--(4,8.5);
\draw (2,7.5)--(2,8.5);
\draw [ultra thick](1,6.5)--(1.5,7);
\draw (1.5,7)--(2,7.5);
\draw [thick](1,7.5)--(1.4,7.1);
\draw [thick](1.6,6.9)--(2,6.5);
\draw (2,2.7)node[left]{$a_1$};
\draw (3,10.5)node{$\nu_1$};
\draw (1,8)node[left]{$\iota$};
\draw (2,8)node[right]{$\bar\lambda$};
\draw (4,8)node[right]{$\nu_2$};
\draw (1,5)node[left]{$\alpha_{\bar\lambda}^+$};
\draw (3,4.7)node[right]{$a_2$};
\end{tikzpicture}
\caption{Redrawing of Fig.~\ref{conn7}}
\label{conn8}
\end{center}
\end{figure}

The complex number represented 
by the diagram in Fig.~\ref{conn8} is
equal to the one represented by the connection diagram
in Fig.~\ref{conn9} multiplied by
$\sqrt{d_\lambda d_\iota d_{\nu_2} d_{a_1}}$
This means that the value of the connection diagram
in Fig.\ref{conn9} for $W_3(\bar\lambda)$
is equal to the complex conjugate of the
value of the connection diagram in Fig.\ref{conn4} for
$W_3(\lambda)$ multiplied by
$\displaystyle\sqrt{\frac{d_{\nu_1}d_{a_2}}{d_{\nu_2}d_{a_1}}}$.
This proves bi-unitarity of $W_3(\lambda)$, as
Fig.~\ref{conn16} gave bi-unitarity of $W_1(\lambda,\mu)$.

\begin{figure}[H]
\begin{center}
\begin{tikzpicture}
\draw [-to](1,1)--(2,1);
\draw [-to](1,2)--(2,2);
\draw [-to](1,2)--(1,1);
\draw [-to](2,2)--(2,1);
\draw (1,1)node[below left]{$a_2$};
\draw (2,1)node[below right]{$a_1$};
\draw (1,2)node[above left]{$\nu_2$};
\draw (2,2)node[above right]{$\nu_1$};
\end{tikzpicture}
\caption{The standard diagram for a connection $W_3(\bar\lambda)$}
\label{conn9}
\end{center}
\end{figure}


We next introduce the connection $W_4(\alpha^+_\lambda,\mu)$,
which is the $\alpha^+$-induced connection.

\begin{definition}\label{W4}{\rm
Let $a_1,a_2,a_3,a_4$ be irreducible $M$-$N$ morphisms
arising from $\Delta$ and the subfactor $N\subset M$.
Consider the diagram in Fig.~\ref{conn2x}.
By composing isometries $T_1\in\Hom(a_1\mu,a_2)$,
$T_2\in\Hom(\alpha_\lambda^+ a_2,a_4)$,
$T_3\in\Hom(\alpha_\lambda^+ a_1,a_3)$, and
$T_4\in\Hom(a_3\mu,a_4)$, we obtain a complex number
$T_4 T_3 \alpha_\lambda^+(T_1^*)T_2^*\in\Hom(a_4,a_4)$.
We define the connection $W_4(\alpha^+_\lambda,\mu)$ by 
this number and represent this as in Fig.~\ref{conn11}.
}\end{definition}

\begin{figure}[H]
\begin{center}
\begin{tikzpicture}[scale=2.5]
\draw [-to](1.2,1)--(1.8,1);
\draw [-to](1.2,2)--(1.8,2);
\draw [-to](1,1.8)--(1,1.2);
\draw [-to](2,1.8)--(2,1.2);
\draw (1,1)node{$a_3\mu$};
\draw (2,1)node{$a_4$};
\draw (1,2)node{$\alpha_\lambda^+ a_1\mu$};
\draw (2,2)node{$\alpha_\lambda^+ a_2$};
\draw (1.5,2)node[above]{$\alpha_\lambda^+(T_1)$};
\draw (1.5,1)node[below]{$T_4$};
\draw (1,1.5)node[left]{$T_3$};
\draw (2,1.5)node[right]{$T_2$};
\end{tikzpicture}
\caption{The diagram for the connection $W_4(\alpha_\lambda^+,\mu)$
in Fig.~\ref{conn1}}
\label{conn2x}
\end{center}
\end{figure}

\begin{figure}[H]
\begin{center}
\begin{tikzpicture}
\draw [-to](1,1)--(2,1);
\draw [-to](1,2)--(2,2);
\draw [-to](1,2)--(1,1);
\draw [-to](2,2)--(2,1);
\draw (1,1)node[below left]{$a_3$};
\draw (2,1)node[below right]{$a_4$};
\draw (1,2)node[above left]{$a_1$};
\draw (2,2)node[above right]{$a_2$};
\end{tikzpicture}
\caption{The standard diagram for a connection $W_4(\alpha^+_\lambda,\mu)$}
\label{conn11}
\end{center}
\end{figure}

By a very similar argument to the one for
bi-unitarity of $W_1(\lambda,\mu)$, we obtain 
bi-unitarity of $W_4(\alpha^+_\lambda,\mu)$.


We then have the intertwining Yang-Baxter equation for
$W_1(\lambda,\mu)$, $W_3(\lambda)$, $W_3(\mu)$, and
$W_4(\alpha^+_\lambda,\mu)$ as in Fig.~\ref{conn17}, which
was given in \cite[Axiom 7]{K2}.  The meaning of this
diagram is as follows.  On the both hand sides of
the identity, the six isometries are fixed for the 
six boundary edges of the hexagons in the same way.  The left
hand side means the summation of the product of
the three connection values $W_1(\lambda,\mu)$,
$W_3(\lambda)$, and $W_2(\mu)$ over all possible
choices of the three isometries corresponding to
the three internal edges of the hexagon.   The right
hand side means the summation of the product of
the three connection values 
$W_3(\lambda)$, $W_2(\mu)$, and 
$W_4(\alpha^+_\lambda,\mu)$ over all possible
choices of the three isometries corresponding to
the three internal edges of the hexagon.  The
both hand sides are equal because they are both equal
to the composition of the six (co-)isometries
corresponding to the six boundary edges of the hexagons.

\begin{figure}[H]
\begin{center}
\begin{tikzpicture}[scale=1.5]
\draw [-to](1,1)--(2.5,1);
\draw [-to](1,2.5)--(2.5,2.5);
\draw [-to](1,2.5)--(1,1);
\draw [-to](2.5,2.5)--(2.5,1);
\draw [-to](2,3.5)--(1,2.5);
\draw [-to](3.5,3.5)--(2.5,2.5);
\draw [-to](3.5,2)--(2.5,1);
\draw [-to](2,3.5)--(3.5,3.5);
\draw [-to](3.5,3.5)--(3.5,2);
\draw (4,2.25)node{$=$};
\draw [-to](4.5,1)--(6,1);
\draw [-to](4.5,2.5)--(4.5,1);
\draw [-to](5.5,3.5)--(4.5,2.5);
\draw [-to](7,2)--(6,1);
\draw [-to](5.5,3.5)--(7,3.5);
\draw [-to](7,3.5)--(7,2);
\draw [-to](5.5,2)--(7,2);
\draw [-to](5.5,3.5)--(5.5,2);
\draw [-to](5.5,2)--(4.5,1);
\draw (1.75,1.75)node{$W_2(\mu)$};
\draw (2.25,3)node{$W_1(\lambda,\mu)$};
\draw (3,2.25)node{$W_3(\lambda)$};
\draw (5,2.25)node{$W_3(\lambda)$};
\draw (5.75,1.5)node{$W_4(\alpha^+_\lambda,\mu)$};
\draw (6.25,2.75)node{$W_2(\mu)$};
\end{tikzpicture}
\caption{The intertwining Yang-Baxter equation for
$W_1(\lambda,\mu)$, $W_3(\lambda)$, $W_3(\mu)$, and
$W_4(\alpha^+_\lambda,\mu)$}
\label{conn17}
\end{center}
\end{figure}


We have seen $\alpha$-induction applied to $N\subset M$
produces bi-unitary connections.  From this construction
and the description of intertwiner spaces for connections
as in \cite[Theorem 3]{AH} again, we have the following theorem
under Assumption \ref{assump} with the same arguments as in
the proof of Theorem \ref{equiv0}.

\begin{theorem}
\label{equiv}
Under Assumptions \ref{assum0} and \ref{assump} for a fixed $\mu$,
the fusion category arising from the 
connections $W_4(\alpha_\lambda^\pm,\mu)$
for all $\lambda\in\Delta$
is equivalent to the one arising from 
endomorphisms $\alpha_\lambda^\pm$
of $M$ given by $\alpha$-induction.
\end{theorem}

Note that it is nontrivial that we have only finite many
irreducible endomorphisms of $M$ up to equivalence when
we consider those arising from $\alpha_\lambda^\pm$.  This
has been proved in  \cite[Theorem 5.10]{BEK1}.

\section{Rewriting $W_4(\alpha^+_\lambda,\mu)$ and switching of
positive and negative braiding}

Though we have defined the $\alpha$-induced connection
$W_4(\alpha^+_\lambda,\mu)$ as in Fig. \ref{conn2x}, we
need to know full information about $\alpha^+_\lambda$
to compute this value.  Since we are now in the process
of defining the new connection $W_4(\alpha^+_\lambda,\mu)$
before knowing $\alpha^+_\lambda$, we certainly hope to
compute this number without using the information about
$\alpha^+_\lambda$.  We show in this Section that this is possible.

Fig.~\ref{conn10} represents the complex number
\[
T_4 T'_3\E^+(\lambda,a_1) \alpha_\lambda^+(T_1)^* 
\E^+(\lambda,a_2)^* {T'_2}^*\in\Hom(a_4,a_4)
\]
for isometries $T_1\in\Hom(a_1\mu,a_2)$,
$T'_2\in\Hom(a_2\lambda,a_4)$,
$T'_3\in\Hom(a_1\lambda,a_3)$,
$T_4\in\Hom(a_3\mu,a_4)$, which is represented with
the connection diagram in Fig.~\ref{conn11}.
This composition corresponds to the diagram in
Fig.~\ref{conn12} up to normalization constant, and
we can redraw this as in Fig.~\ref{conn13}.
Note that we drop orientations of wires which go
from the top to the bottom in Fig.~\ref{conn12}, \ref{conn13}. 
The complex number represented by the diagram in Fig.~\ref{conn13}
is equal to the one represented by the connection diagram in
Fig.~\ref{conn10} multiplied by
$\sqrt{d_\lambda d_\mu  d_{a_1} d_{a_4}}$.
Since Fig.~\ref{conn13} does not involve $\alpha_\lambda^+$, this
diagram gives $W_4(\alpha_\lambda^+,\mu)$, up to normalization
constant, in terms of $N$-$N$ and $M$-$N$ morphisms and
intertwiners including the braiding operators.  This is
what we asked for at the beginning of this Section.

\begin{figure}[H]
\begin{center}
\begin{tikzpicture}[scale=2.5]
\draw [-to](1.2,1)--(1.8,1);
\draw [-to](1.2,2)--(1.8,2);
\draw [-to](1,1.8)--(1,1.7);
\draw [-to](1,1.3)--(1,1.2);
\draw [-to](2,1.8)--(2,1.7);
\draw [-to](2,1.3)--(2,1.2);
\draw (1,1)node{$a_3\mu$};
\draw (1,1.5)node{$a_1\lambda\mu$};
\draw (1,2)node{$\alpha_\lambda^+ a_1\mu$};
\draw (2,1)node{$a_4$};
\draw (2,1.5)node{$a_2\lambda $};
\draw (2,2)node{$\alpha_\lambda^+ a_2$};
\draw (1.5,2)node[above]{$\alpha_\lambda^+(T_1)$};
\draw (1.5,1)node[below]{$T_4$};
\draw (1,1.75)node[left]{$\E^+(\lambda,a_1)$};
\draw (1,1.25)node[left]{$T'_3$};
\draw (2,1.75)node[right]{$\E^+(\lambda,a_2)$};
\draw (2,1.25)node[right]{$T'_2$};
\end{tikzpicture}
\caption{The diagram for the connection $W_4(\alpha^+_\lambda,\mu)$}
\label{conn10}
\end{center}
\end{figure}

\begin{figure}[H]
\begin{center}
\begin{tikzpicture}[scale=0.6]
\draw [thick](6,13.5) arc (0:180:1.75);
\draw [thick](2.5,13.5) arc (90:180:1.5);
\draw (4,12) arc (0:90:1.5);
\draw (3.5,2.5) arc (270:360:1.5);
\draw (5,7) arc (0:90:1);
\draw (2,4) arc (270:360:1);
\draw [thick](4,8) arc (90:180:1);
\draw [thick](2,4) arc (180:270:1.5);
\draw [thick](3.5,2.5) arc (180:360:1.25);
\draw [thick](1,5) arc (180:270:1);
\draw [thick](1,5)--(1.9,5.9);
\draw [thick](2.1,6.1)--(3,7);
\draw [thick](1,12)--(2.4,10.6);
\draw [thick](2.6,10.4)--(4,9);
\draw [ultra thick](2,6)--(1,7);
\draw [ultra thick](1,7)--(1,9);
\draw [ultra thick](1,9)--(2.5,10.5);
\draw (2.5,10.5)--(4,12);
\draw (5,4)--(5,7);
\draw [thick](4,8)--(4,9);
\draw (3,5)--(2,6);
\draw [thick](6,2.5)--(6,13.5);
\draw (4,13.5)node{$\lambda$};
\draw (1,13.5)node{$a_2$};
\draw (1,8)node[left]{$\alpha_\lambda^+$};
\draw (1,4)node{$a_1$};
\draw (3,4)node{$\lambda$};
\draw (5,4)node[left]{$\mu$};
\draw (2,2.5)node{$a_3$};
\draw (3.5,1.25)node{$a_4$};
\end{tikzpicture}
\caption{The diagram for the connection in Fig.~\ref{conn10}}
\label{conn12}
\end{center}
\end{figure}

\begin{figure}[H]
\begin{center}
\begin{tikzpicture}[scale=0.8]
\draw [thick](5,7) arc (0:180:1);
\draw [thick](3,7) arc (90:180:1);
\draw [thick](2,6) arc (90:180:1);
\draw [thick](1,4) arc (180:270:1);
\draw [thick](2,3) arc (180:270:1);
\draw [thick](3,2) arc (180:360:1);
\draw [thick](1,4)--(1,5);
\draw [thick](5,2)--(5,7);
\draw (4,6) arc (0:90:1);
\draw (3,5) arc (0:90:1);
\draw (2,3) arc (270:360:1);
\draw (3,2) arc (270:360:1);
\draw (4,3)--(4,4);
\draw (4,5)--(4,6);
\draw (3,4)--(4,5);
\draw (3,5)--(3.4,4.6);
\draw (3.6,4.4)--(4,4);
\draw (3,8)node{$a_4$};
\draw (2,7)node{$a_2$};
\draw (1,4.5)node[left]{$a_1$};
\draw (4,3.5)node[right]{$\mu$};
\draw (4,5.5)node[right]{$\lambda$};
\draw (2,2)node{$a_3$};
\end{tikzpicture}
\caption{Redrawing of Fig.~\ref{conn12}}
\label{conn13}
\end{center}
\end{figure}


We next study what the effect of switching the positive
and negative braiding is.  Vertical reflection of
Fig.~\ref{conn13} gives Fig.~\ref{conn14a}.  The complex
numbers given by these two Figures are mutually complex
conjugate.  By comparing these two Figures, we have the
following Proposition.

\begin{figure}[H]
\begin{center}
\begin{tikzpicture}[scale=0.8]
\draw [thick](5,7) arc (0:180:1);
\draw [thick](3,7) arc (90:180:1);
\draw [thick](2,6) arc (90:180:1);
\draw [thick](1,4) arc (180:270:1);
\draw [thick](2,3) arc (180:270:1);
\draw [thick](3,2) arc (180:360:1);
\draw [thick](1,4)--(1,5);
\draw [thick](5,2)--(5,7);
\draw (4,6) arc (0:90:1);
\draw (3,5) arc (0:90:1);
\draw (2,3) arc (270:360:1);
\draw (3,2) arc (270:360:1);
\draw (4,3)--(4,4);
\draw (4,5)--(4,6);
\draw (3,4)--(3.4,4.4);
\draw (3.6,4.6)--(4,5);
\draw (3,5)--(4,4);
\draw (3,8)node{$a_4$};
\draw (2,7)node{$a_3$};
\draw (1,4.5)node[left]{$a_1$};
\draw (4,3.5)node[right]{$\lambda$};
\draw (4,5.5)node[right]{$\mu$};
\draw (2,2)node{$a_2$};
\end{tikzpicture}
\caption{Vertical reflection of Fig.~\ref{conn13}}
\label{conn14a}
\end{center}
\end{figure}

\begin{proposition}\label{conj}
We have the identity as in Fig.~\ref{conjugate}.
\end{proposition}

\begin{figure}[H]
\begin{center}
\begin{tikzpicture}[scale=2]
\draw [-to](1,1)--(2,1);
\draw [-to](1,2)--(2,2);
\draw [-to](1,2)--(1,1);
\draw [-to](2,2)--(2,1);
\draw [thick] (0.7,2.3)--(2.3,2.3);
\draw (1,1)node[below left]{$a_3$};
\draw (1,2)node[above left]{$a_1$};
\draw (2,1)node[below right]{$a_4$};
\draw (2,2)node[above right]{$a_2$};
\draw (1.5,1.5)node{$W_4(\alpha_\lambda^+,\mu)$};
\draw (2.5,1.5)node{$=$};
\draw (3.5,1.5)node{$W_4(\alpha_\mu^-,\lambda)$};
\draw [-to](3,1)--(4,1);
\draw [-to](3,2)--(4,2);
\draw [-to](3,2)--(3,1);
\draw [-to](4,2)--(4,1);
\draw (3,1)node[below left]{$a_2$};
\draw (3,2)node[above left]{$a_1$};
\draw (4,1)node[below right]{$a_4$};
\draw (4,2)node[above right]{$a_3$};
\end{tikzpicture}
\caption{Complex conjugate and opposite braiding}
\label{conjugate}
\end{center}
\end{figure}

We also show that we can rewrite $W_3(\lambda)$ into the
form without using $\alpha_\lambda^+$ as follows.

\begin{figure}[H]
\begin{center}
\begin{tikzpicture}[scale=0.6]
\draw [thick](6,13.5) arc (0:180:1.75);
\draw (4,12) arc (0:90:1.5);
\draw [thick](2.5,13.5) arc (90:180:1.5);
\draw (5,10) arc (0:180:1);
\draw [thick](1,5) arc (180:270:1);
\draw (2,4) arc (270:360:1);
\draw [thick](2,3) arc (180:360:2);
\draw [thick](2,3)--(2,4);
\draw [thick](6,3)--(6,13.5);
\draw [thick](1,10)--(1,12);
\draw [thick](3,7)--(3,8);
\draw [thick](1,10)--(1.9,9.1);
\draw [thick](2.1,8.9)--(3,8);
\draw [thick](1,5)--(1.9,5.9);
\draw [thick](2.1,6.1)--(3,7);
\draw (3,5)--(2,6);
\draw (2,9)--(3,10);
\draw (3,8)--(5,10);
\draw (4,11)--(4,12);
\draw [ultra thick](2,9)--(1,8);
\draw [ultra thick](1,8)--(1,7);
\draw [ultra thick](1,7)--(2,6);
\draw (4,11.5)node[right]{$\nu_2$};
\draw (1,11.5)node[left]{$\iota$};
\draw (3,11)node{$\lambda$};
\draw (4,9)node[below right]{$\nu_1$};
\draw (3,7.3)node[right]{$a_1$};
\draw (1,7.5)node[left]{$\alpha_\lambda^+$};
\draw (3,5.6)node{$\lambda$};
\draw (2,3.3)node[left]{$a_2$};
\end{tikzpicture}
\caption{Redrawing of $W_3(\lambda)$}
\label{conn18}
\end{center}
\end{figure}

\begin{figure}[H]
\begin{center}
\begin{tikzpicture}[scale=0.6]
\draw [thick](6,12) arc (0:180:1.75);
\draw (4,10.5) arc (0:90:1.5);
\draw [thick](2.5,12) arc (90:180:1.5);
\draw [thick](1,6.5) arc (180:270:1);
\draw (5,8.5) arc (0:180:1);
\draw [thick](2,4.5) arc (180:270:1);
\draw [thick](3,2.5) arc (180:360:1.5);
\draw (3,3.5) arc (270:360:1);
\draw (2,5.5) arc (270:360:1);
\draw [thick](2,4.5)--(2,5.5);
\draw [thick](3,2.5)--(3,3.5);
\draw [thick](1,6.5)--(1,10.5);
\draw [thick](6,2.5)--(6,12);
\draw (4,4.5)--(5,6.5);
\draw (5,6.5)--(3,8.5);
\draw (4,9.5)--(4,10.5);
\draw (3,6.5)--(3.9,7.4);
\draw (4.1,7.6)--(5,8.5);
\draw (4,10.5)node[right]{$\nu_2$};
\draw (1,9)node[left]{$\iota$};
\draw (4.7,5.5)node[below]{$\lambda$};
\draw (3.1,5.5)node{$\nu_1$};
\draw (2,4.7)node[left]{$a_1$};
\draw (3,2.5)node[left]{$a_2$};
\end{tikzpicture}
\caption{Further redrawing of $W_3(\lambda)$}
\label{conn19}
\end{center}
\end{figure}

\section{Locality of the $Q$-system 
and flatness of $W_4(\alpha_\lambda^+,\mu)$}

In the case of Ocneanu's construction \cite{O3}, he obtained flat
connections only for $A_n$, $D_{2n}$, $E_6$ and $E_8$ as we see 
in the next Section in detail.  (Note that Assumption \ref{assump}
does not hold for these examples.  We will treat this issue in
the following Section.  Also see \cite[Theorem 11.24]{EK2}.)
It has been observed that these cases exactly correspond
to the $Q$-systems with locality $\e(\theta,\theta)\gamma(v)
=\gamma(v)$ as in \cite[Section 5]{BEK2} (where this property was
called chiral locality).  So we expect some
relations between locality of the $Q$-system and
flatness of the corresponding $\alpha^\pm$-induced connection in general.
We show that this is indeed the case in this Section.

We now assume locality.  Note that we then have
irreducibility of $\iota:N \hookrightarrow M$ by
\cite[Corollary 3.6]{BE1}.  (Having a $Q$-system with
locality in a general modular tensor category is enough in \cite{BE1}
rather than a conformal net.)

We prove flatness of the connection $W_4(\alpha_\lambda^+,\mu)$
in the sense of Fig.~\ref{f1121}, which is taken from
\cite[Fig.~11.21]{EK2}.  We first need a lemma.

\begin{figure}[H]
\begin{center}
\begin{tikzpicture}[scale=0.8]
\draw [-to](1,3.5)--(1,2.5);
\draw [-to](4,3.5)--(4,2.5);
\draw [-to](1,3.5)--(2,3.5);
\draw [-to](1,1)--(2,1);
\filldraw (2.5,3.5) circle (0.7pt);
\filldraw (3,3.5) circle (0.7pt);
\filldraw (3.5,3.5) circle (0.7pt);
\filldraw (2.5,1) circle (0.7pt);
\filldraw (3,1) circle (0.7pt);
\filldraw (3.5,1) circle (0.7pt);
\filldraw (1,1.5) circle (0.7pt);
\filldraw (1,2) circle (0.7pt);
\filldraw (4,1.5) circle (0.7pt);
\filldraw (4,2) circle (0.7pt);
\draw (1,3.5)node{$*$};
\draw (1,1)node{$*$};
\draw (4,3.5)node{$*$};
\draw (4,1)node{$*$};
\draw (5,2.25)node{$=1$};
\draw (1,2.75)node[left]{$\rho$};
\draw (4,2.75)node[right]{$\rho$};
\draw (2.5,3.5)node[above]{$\sigma$};
\draw (2.5,1)node[below]{$\sigma$};
\end{tikzpicture}
\caption{Flatness as in \cite[Fig.~11.21]{EK2}}
\label{f1121}
\end{center}
\end{figure}

\begin{lemma}\label{cross}
Suppose $\lambda,\mu$ are irreducible endomorphisms of $N$
contained in $\theta$.
We then have the identity as in Fig.~\ref{cross1}.
\end{lemma}

\begin{figure}[H]
\begin{center}
\begin{tikzpicture}[scale=0.8]
\draw [thick](1,4)--(1,1);
\draw [thick](4.5,4)--(4.5,1);
\draw (1,1.5)--(3,3.5);
\draw (4.5,1.5)--(6.5,1.5);
\draw (4.5,3.5)--(6.5,3.5);
\draw (1,3.5)--(1.8,2.7);
\draw (2.2,2.3)--(3,1.5);
\draw (1,4)[left]node{$\iota$};
\draw (4.5,4)[left]node{$\iota$};
\draw (3.75,2.5)node{$=$};
\draw (6.5,1.5)node[right]{$\lambda$};
\draw (6.5,3.5)node[right]{$\mu$};
\draw (3,1.5)node[right]{$\lambda$};
\draw (3,3.5)node[right]{$\mu$};
\end{tikzpicture}
\caption{A consequence of locality}
\label{cross1}
\end{center}
\end{figure}

\begin{proof}
Locality gives Fig.~\ref{loc1}, where the triple
vertices on the both hand sides represent $\gamma(v)$.
We then have the
identity in Fig.~\ref{loc2}, where small black circles represent
co-isometries in $\Hom(\theta,\lambda)$ and $\Hom(\theta,\mu)$.
We next have the identity in Fig.~\ref{loc3} by a property
of braiding and then 
the identity in Fig.~\ref{loc4} by rewriting $\theta$ with
$\bar\iota\iota$.  Pulling the thick wires straight and
rotating this diagram for
$90$ degrees gives the desired conclusion.
\end{proof}

\begin{figure}[H]
\begin{center}
\begin{tikzpicture}[scale=0.8]
\draw (1.5,2.5)--(1.5,3.5);
\draw (2,2) arc (0:180:0.5);
\draw (1,2)--(2,1);
\draw (1,1)--(1.3,1.3);
\draw (1.7,1.7)--(2,2);
\draw (4,2)--(4,3.5);
\draw (4,2)--(3.5,1);
\draw (4,2)--(4.5,1);
\draw (2.75,2)node{$=$};
\draw (1,1)node[below]{$\theta$};
\draw (2,1)node[below]{$\theta$};
\draw (3.5,1)node[below]{$\theta$};
\draw (4.5,1)node[below]{$\theta$};
\draw (1.5,3.5)node[above]{$\theta$};
\draw (4,3.5)node[above]{$\theta$};
\end{tikzpicture}
\caption{Locality of a $Q$-system}
\label{loc1}
\end{center}
\end{figure}

\begin{figure}[H]
\begin{center}
\begin{tikzpicture}[scale=0.8]
\draw (1.5,3.5)--(1.5,4.5);
\draw (2,3) arc (0:180:0.5);
\draw (1,3)--(2,2);
\draw (1,2)--(1.3,2.3);
\draw (1.7,2.7)--(2,3);
\draw (4,3)--(4,4.5);
\draw (4,3)--(3.5,2);
\draw (4,3)--(4.5,2);
\draw (1,1)--(1,2);
\draw (2,1)--(2,2);
\draw (3.5,1)--(3.5,2);
\draw (4.5,1)--(4.5,2);
\draw (2.75,2.5)node{$=$};
\draw (1,3)node[left]{$\theta$};
\draw (2,3)node[right]{$\theta$};
\draw (3.5,2.5)node{$\theta$};
\draw (4.5,2.5)node{$\theta$};
\draw (1.5,4.5)node[above]{$\theta$};
\draw (4,4.5)node[above]{$\theta$};
\draw (1,1)node[below]{$\lambda$};
\draw (2,1)node[below]{$\mu$};
\draw (3.5,1)node[below]{$\lambda$};
\draw (4.5,1)node[below]{$\mu$};
\filldraw (1,2) circle (2pt);
\filldraw (2,2) circle (2pt);
\filldraw (3.5,2) circle (2pt);
\filldraw (4.5,2) circle (2pt);
\end{tikzpicture}
\caption{A consequence of Fig.~\ref{loc1}}
\label{loc2}
\end{center}
\end{figure}

\begin{figure}[H]
\begin{center}
\begin{tikzpicture}[scale=0.8]
\draw (1.5,2.5)--(1.5,3.5);
\draw (2,2) arc (0:180:0.5);
\draw (1,2)--(2,1);
\draw (1,1)--(1.3,1.3);
\draw (1.7,1.7)--(2,2);
\draw (4,2.5)--(4,3.5);
\draw (4,2.5)--(3.5,1.5);
\draw (4,2.5)--(4.5,1.5);
\draw (3.5,1)--(3.5,1.5);
\draw (4.5,1)--(4.5,1.5);
\draw (2.75,2)node{$=$};
\draw (1,2.5)node[left]{$\theta$};
\draw (2,2.5)node[right]{$\theta$};
\draw (3.5,2)node{$\theta$};
\draw (4.5,2)node{$\theta$};
\draw (1.5,3.5)node[above]{$\theta$};
\draw (4,3.5)node[above]{$\theta$};
\draw (1,1)node[below]{$\lambda$};
\draw (2,1)node[below]{$\mu$};
\draw (3.5,1)node[below]{$\lambda$};
\draw (4.5,1)node[below]{$\mu$};
\filldraw (1,2) circle (2pt);
\filldraw (2,2) circle (2pt);
\filldraw (3.5,1.5) circle (2pt);
\filldraw (4.5,1.5) circle (2pt);
\end{tikzpicture}
\caption{A consequence of Fig.~\ref{loc2}}
\label{loc3}
\end{center}
\end{figure}

\begin{figure}[H]
\begin{center}
\begin{tikzpicture}[scale=0.8]
\draw [thick](1.2,2.5)--(1.2,3.5);
\draw [thick](1.8,2.5)--(1.8,3.5);
\draw [thick](1.2,2.5) arc (90:180:0.5);
\draw [thick](2,2) arc (0:180:0.5);
\draw [thick](2.3,2) arc (0:90:0.5);
\draw [thick](0.7,2) arc (180:360:0.15);
\draw [thick](2,2) arc (180:360:0.15);
\draw (0.85,1.85)--(2.15,0.55);
\draw (0.85,0.55)--(1.35,1.05);
\draw (1.55,1.35)--(2.15,1.85);
\draw (2.75,2)node{$=$};
\draw (1.2,3.5)node[above]{$\bar\iota$};
\draw (1.8,3.5)node[above]{$\iota$};
\draw (1.2,1.8)node[above]{$\iota$};
\draw (0.85,0.55)node[below]{$\lambda$};
\draw (2.15,0.55)node[below]{$\mu$};
\draw [thick](3.7,2.5)--(3.7,3.5);
\draw [thick](4.3,2.5)--(4.3,3.5);
\draw [thick](3.7,2.5) arc (90:180:0.5);
\draw [thick](4.5,2) arc (0:180:0.5);
\draw [thick](4.8,2) arc (0:90:0.5);
\draw [thick](3.2,2) arc (180:360:0.15);
\draw [thick](4.5,2) arc (180:360:0.15);
\draw (3.35,1.85)--(3.35,0.55);
\draw (4.65,1.85)--(4.65,0.55);
\draw (3.7,3.5)node[above]{$\bar\iota$};
\draw (4.3,3.5)node[above]{$\iota$};
\draw (3.7,1.8)node[above]{$\iota$};
\draw (3.35,0.55)node[below]{$\lambda$};
\draw (4.65,0.55)node[below]{$\mu$};
\end{tikzpicture}
\caption{A consequence of Fig.~\ref{loc3}}
\label{loc4}
\end{center}
\end{figure}

\begin{theorem}\label{flat}
Under Assumption \ref{assum0},
the connection $W_4(\alpha_\lambda^+,\mu)$ is flat.
\end{theorem}

\begin{proof}
We prove the identity in Fig.~\ref{f1121}. 
Both horizontal and vertical sizes of the large
diagram in Fig.~\ref{f1121}, are supposed to be even, but
we can take both of them to be $1$ in our current setting, 
so we first give a proof for this case.

Now our $*$ vertex in Fig.~\ref{f1121} corresponds to
$\iota$.  Because all the four corner vertices in
Fig.~\ref{f1121} are now $\iota$, we
set $a_1=a_2=a_3=a_4=\iota$ in Fig.~\ref{conn13}.  We then have
Fig.~\ref{conn20} and the complex value this diagram represents
is equal to the one given by the partition function in 
Fig.~\ref{f1121} multiplied by
$\sqrt{d_\lambda d_\mu}d_\iota$.

By Lemma \ref{cross}, the value Fig.~\ref{conn20} represents
is equal to the one Fig.~\ref{conn21} represents.  The latter
is equal to $\sqrt{d_\lambda d_\mu}d_\iota$, so the complex
value given by the partition function in 
Fig.~\ref{f1121} is $1$.

When the horizontal and vertical sizes of the large
diagram in Fig.~\ref{f1121} are arbitrary, we replace
$\lambda$ and $\mu$ in the above arguments by
$\bar\lambda\lambda\cdots\lambda$ and
$\mu\bar\mu\cdots\bar\mu$.  Then the same arguments give
the value $1$ and we are done.
\end{proof}

\begin{figure}[H]
\begin{center}
\begin{tikzpicture}[scale=0.8]
\draw [thick](5,7) arc (0:180:1);
\draw [thick](3,7) arc (90:180:1);
\draw [thick](2,6) arc (90:180:1);
\draw [thick](1,4) arc (180:270:1);
\draw [thick](2,3) arc (180:270:1);
\draw [thick](3,2) arc (180:360:1);
\draw [thick](1,4)--(1,5);
\draw [thick](5,2)--(5,7);
\draw (4,6) arc (0:90:1);
\draw (3,5) arc (0:90:1);
\draw (2,3) arc (270:360:1);
\draw (3,2) arc (270:360:1);
\draw (4,3)--(4,4);
\draw (4,5)--(4,6);
\draw (3,4)--(4,5);
\draw (3,5)--(3.4,4.6);
\draw (3.6,4.4)--(4,4);
\draw (3,8)node{$\iota$};
\draw (2,7)node{$\iota$};
\draw (1,4.5)node[left]{$\iota$};
\draw (4,3.5)node[right]{$\mu$};
\draw (4,5.5)node[right]{$\lambda$};
\draw (2,2)node{$\iota$};
\end{tikzpicture}
\caption{A diagram for the connection $W_4(\alpha_\lambda^+,\mu)$
for $a_1=a_2=a_3=a_4=\iota$}
\label{conn20}
\end{center}
\end{figure}

\begin{figure}[H]
\begin{center}
\begin{tikzpicture}[scale=0.8]
\draw [thick](3,9) arc (0:180:1);
\draw [thick](1,2) arc (180:360:1);
\draw (2,7) arc (0:90:1);
\draw (1,6) arc (270:360:1);
\draw (2,4) arc (0:90:1);
\draw (1,3) arc (270:360:1);
\draw [thick](1,2)--(1,9);
\draw [thick](3,2)--(3,9);
\draw (1,8.5)node[left]{$\iota$};
\draw (1,7)node[left]{$\iota$};
\draw (1,5.5)node[left]{$\iota$};
\draw (1,4)node[left]{$\iota$};
\draw (2,7)node[right]{$\mu$};
\draw (2,4)node[right]{$\lambda$};
\end{tikzpicture}
\caption{Redrawing of Fig.~\ref{conn20} with Lemma \ref{cross}}
\label{conn21}
\end{center}
\end{figure}

\section{Examples}\label{exa}

Ocneanu considered connections on
$A$-$D$-$E$ Dynkin diagrams in \cite{O3}.  We revisit
this topic from our viewpoint now and would like to apply
the results in Section \ref{alpha}.
Here we have a nontrivial issue since Assumption \ref{assump}
does \textit{not} hold now for any choice of $\mu$.

Let $\Delta$ be the set of endomorphisms of a type III
factor $N$ corresponding
to the Wess-Zumino-Witten model $SU(2)_k$, where $k$ is a 
positive integer called a \textit{level}.  (See 
\cite[Subsection 16.2.3]{DMS}, for example.)  We label
the irreducible objects of the modular tensor category
with $0,1,2,\dots,k$, using the Dynkin diagram of type
$A_{k+1}$ as in Fig.~\ref{SU2}, where the label $0$
denotes the vacuum representation, that is, the
identity automorphism of $N$.  Such a system of
endomorphisms with braiding has been constructed from a
conformal net by
Wassermann \cite{W} and this braiding is nondegenerate 
by \cite[Corollary 37]{KLM} and \cite[Theorem 4.1]{X2}.

\begin{figure}[H]
\begin{center}
\begin{tikzpicture}[scale=0.8]
\draw [thick] (1,1)--(4,1);
\draw [thick] (6,1)--(7,1);
\filldraw (4.5,1) circle (1pt);
\filldraw (5,1) circle (1pt);
\filldraw (5.5,1) circle (1pt);
\filldraw (1,1) circle (3pt);
\filldraw (2,1) circle (3pt);
\filldraw (3,1) circle (3pt);
\filldraw (4,1) circle (3pt);
\filldraw (6,1) circle (3pt);
\filldraw (7,1) circle (3pt);
\draw (1,1)node[below]{$0$};
\draw (2,1)node[below]{$1$};
\draw (3,1)node[below]{$2$};
\draw (4,1)node[below]{$3$};
\draw (6,1)node[below]{$k-1$};
\draw (7,1)node[below]{$k$};
\end{tikzpicture}
\end{center}
\caption{Irreducible objects for the $SU(2)_k$ WZW-model}
\label{SU2}
\end{figure}

All $Q$-systems on $\Delta$ for all $k$ have been classified
in \cite[Section 2.5]{KL} for the local case and in
\cite[Section 2]{KLPR} for the general case.  They
are classified with a pair consisting one of the 
$A$-$D$-$E$ Dynkin diagrams and its vertex.  We explain
this from a viewpoint of the Goodman-de la Harpe-Jones
subfactor in \cite[Section 4.5]{GHJ}.

Let $\G$ be one of the $A$-$D$-$E$ Dynkin diagrams and choose
the vertex with the smallest
entry of the Perron-Frobenius eigenvector entry.  
We then have the Goodman-de la Harpe-Jones
subfactor as in \cite[Section 4.5]{GHJ} or
\cite[Section 11.6]{EK2}, and the corresponding $Q$-system.
If $\G$ is of type $A$, then this $Q$-system has index $1$
and is trivial.  If $\G$ is of type $D$, then the $Q$-system
has index $2$ and corresponds to a crossed product by
$\mathbb{Z}/2\mathbb{Z}$.  If $\G$ is $E_6$ or $E_8$, then
the $Q$-systems correspond to the
subfactors arising from conformal embeddings
$SU(2)_{10}\subset S(5)_1$ and $SU(2)_{28}\subset (\mathrm{G}_2)_1$
by \cite[Proposition A.3]{BEK2}.  Also see \cite[Appendix]{BEK2}
for the case of $E_7$.

We next choose $\lambda=\mu=1$
in the setting of Section \ref{alpha} and consider the connection
$W_4(\alpha_\lambda^\pm,\mu)$.  Since the irreducible $N$-$M$ 
morphisms are labeled with the vertices of $\G$
by the arguments in \cite[Section 11.6]{EK2}, all the four vertex
sets for $W_4(\alpha_\lambda^\pm,\mu)$ are also labeled with 
the vertices of $\G$.  Then the requirements for the
Perron-Frobenius eigenvalues and the Perron-Frobenius eigenvector
entries force all the four graphs of $W_4(\alpha_\lambda^\pm,\mu)$
to be $\G$, but Assumption \ref{assump} does not
hold, since the horizontal top and bottom graphs are never
connected, due to a $\mathbb{Z}/2\mathbb{Z}$-grading on
the vertices of the Dynkin diagrams.  This issue is resolved
as follows.

The connection $W_4(\alpha_\lambda^\pm,\mu)$ splits into two
connections on mutually disjoint graphs, both of which are isomorphic
to $\G$.  Then both of the two connections 
must be of the following form as
given in \cite{O2}.  (Also see \cite[Fig.~11.32]{EK2}.)
Let $n$ be its Coxeter number of $\G$
and set $\displaystyle\varepsilon=
\sqrt{-1}\exp\frac{\pi\sqrt{-1}}{2(n+1)}$.
We write $\mu_x$ for the Perron-Frobenius eigenvector
entry for a vertex $x$.  Then our connection is
given as follows.
(We can replace $\e$ with $\bar\e$.  By the arguments in
\cite[Section 11.5]{EK2}, these two choices give the only possible 
connections on $\G$.  If $\G$ is of type $A$, they give equivalent
connections.  If $\G$ is of type $D$ or $E$, they are not
mutually equivalent.  Here we allow only vertical gauges while
horizontal gauges are also allowed in \cite[Section 11.5]{EK2},
but this does not cause problems since our graph $\G$ is a tree.
See Remark after \cite[Theorem 3]{AH}.)

\begin{figure}[H]
\begin{center}
\begin{tikzpicture}[scale=0.8]
\draw [thick, ->] (1,1)--(2,1);
\draw [thick, ->] (1,2)--(2,2);
\draw [thick, ->] (1,2)--(1,1);
\draw [thick, ->] (2,2)--(2,1);
\draw (1.5,1.5)node{$W$};
\draw (1,1)node[below left]{$l$};
\draw (1,2)node[above left]{$j$};
\draw (2,1)node[below right]{$m$};
\draw (2,2)node[above right]{$k$};
\draw (5,1.5)node{$\displaystyle=\delta_{kl}\varepsilon+
\sqrt{\frac{\mu_k \mu_l}{\mu_j\mu_m}}\delta_{jm}\bar\varepsilon$};
\end{tikzpicture}
\end{center}
\caption{A connection on the Dynkin diagram}
\label{ADEconn}
\end{figure}

The horizontal top and bottom graphs for $W_4(\alpha_j^\pm,1)$ 
always have exactly two
connected components.  This is still valid
after we make irreducible decomposition of such connections.
Let $W$ be a pair of two such connections $W_a$ and $W_b$ arising from 
irreducible decomposition of $W_4(\alpha_j^\pm,1)$ for some $j$ and
$W'$ be a pair of two such connections $W'_a$ and $W'_b$ 
from $W_4(\alpha_k^\pm,1)$ for some $k$.  We can
compose $W_a$ with exactly one of $W'_a$ and $W'_b$, and
compose $W_b$ with the other one,
since they have the matching
horizontal bottom and top graphs.  For the fusion rules
and intertwiner spaces of this composition product, we can use
either of the two compositions and have the same results, 
by the same arguments to
the ones in the proof of Theorem \ref{equiv}.  In this way,
we obtain a fusion category of connections which is equivalent
to the one of $M$-$M$ morphisms arising from $\alpha^\pm$-induction.
We now obtain the diagrams of decompositions of
connections which are the same as in
\cite[Fig.~2, 5, 8, 9]{BE3} and
\cite[Fig.~40, 42]{BEK2}.  These were originally found by Ocneanu for
such connections.  We also know that the results in the previous
Section on flatness apply to these cases, though
Assumption \ref{assump} does not hold now.
For this type of computations, we do not need exact
information of the connections, and simply having the graphs
involved is often sufficient.  See \cite{G} for such
computations. 

We now discuss the issue of complex conjugate connections.
For $W_4(\alpha_1^\pm,1)$,
the connection Fig.~\ref{ADEconn} is symmetric in $j$ and $m$, 
so the effect of switching positive and negative braiding
in Proposition \ref{conj} now
amounts to taking simply complex conjugate connections.  
That is, $W_4(\alpha_1^+,1)$ and $W_4(\alpha_1^-,1)$ are
mutually complex conjugate.
Now consider the case of $\G=E_6$.  The connections
$W_4(\alpha_2^+,1)$ and $W_4(\alpha_2^-,1)$ are
also mutually complex conjugate.  Both of  the connections
$W_4(\alpha_3^+,1)$ and $W_4(\alpha_3^-,1)$ decomposes into
two irreducible connections each.  Since irreducible decomposition
of a complex conjugate connection gives complex conjugate
connections of those appearing in the irreducible decomposition
in general, the complex conjugates of the two irreducible
connections arising from $W_4(\alpha_3^+,1)$ appear in the
irreducible decomposition of $W_4(\alpha_3^-,1)$.  In this
way, we see that switching the positive and negative
braiding for the $\alpha$-induced connections amounts to
taking complex conjugate connections.  The same argument
also works for $E_7$ and $E_8$. This
was also observed by Ocneanu \cite{O3}, but is special to
the $A$-$D$-$E$ Dynkin diagrams.  
The decomposition rules for $E_7$ as in \cite[Fig.~42]{BEK2}
follows from our computations in \cite{EK1} showing that
the principal graph of the subfactor arising from the
$E_7$ connection is $D_{10}$.

The modular tensor categories corresponding to the
Wess-Zumino-Witten model $SU(N)_k$ also have
$\mathbb{Z}/N\mathbb{Z}$-grading, and a similar method
to the above gives how to handle this issue.

We add a remark on the effect of Assumption \ref{assump}.
Even without this Assumption, our 
Definitions \ref{W1}, \ref{W2}, \ref{W3}, \ref{W4} on
our new connections make sense.  The only problem is that
if horizontal graphs are disconnected, we cannot compose
two connections in a usual way in Theorem \ref{equiv}, and 
other results such as Theorem \ref{flat} are not affected.
Another way of handling the issue in Assumption \ref{assump}
is to make $\Delta$ smaller.  For example,
if we choose $\lambda=\mu=2$ in the numbering of irreducible
objects for the fusion category $SU(2)_k$ as in Fig.~\ref{SU2}
and consider only irreducible objects numbered with even integers,
then everything in the above Sections works fine.

\section{Triple sequence of string algebras and
another interpretation of the $\alpha$-induction 
in terms of bi-unitary connections}


Let $*$ be the vertex corresponding to the identity automorphism
of $N$ in $\Delta$.  We construct a triple sequence 
$\{A_{jkl}\}_{jkl}$ as in \cite[Section 2]{K2}.  This is
a triple sequence version of the standard construction 
of the double sequence of string algebras in 
\cite[Section 11.3]{EK2}.  A new property we need for
compatibility of identification is the 
Intertwining Yang-Baxter Equations as in
\cite[Axiom 7]{K2}, and we now have this identity
as in Fig.~\ref{conn17}.

For the commuting squares
\begin{align*}
&\begin{array}{ccc}
A_{2j,2k,2l}&\subset& A_{2j,2k+1,2l}\\
\cap && \cap\\
A_{2j+1,2k,2l}&\subset& A_{2j+1,2k+1,2l}
\end{array},\\
&\begin{array}{ccc}
A_{2j,2k,2l+1}&\subset& A_{2j,2k+1,2l+1}\\
\cap && \cap\\
A_{2j+1,2k,2l+1}&\subset& A_{2j+1,2k+1,2l+1}
\end{array},\\
&\begin{array}{ccc}
A_{2j,2k,2l}&\subset& A_{2j,2k+1,2l}\\
\cap && \cap\\
A_{2j,2k,2l+1}&\subset& A_{2j,2k+1,2l}
\end{array},\\ 
&\begin{array}{ccc}
A_{2j+1,2k,2l}&\subset& A_{2j+1,2k+1,2l}\\
\cap && \cap\\
A_{2j+1,2k,2l+1}&\subset& A_{2j+1,2k+1,2l}
\end{array},\\ 
&\begin{array}{ccc}
A_{2j,2k,2l}&\subset& A_{2j+1,2k,2l}\\
\cap && \cap\\
A_{2j,2k,2l+1}&\subset& A_{2j+1,2k,2l+1}
\end{array},\\
&\begin{array}{ccc}
A_{2j,2k+1,2l}&\subset& A_{2j+1,2k+1,2l}\\
\cap && \cap\\
A_{2j,2k+1,2l+1}&\subset& A_{2j+1,2k+1,2l+1}
\end{array},
\end{align*}
we use the connections
$W_1(\lambda,\mu)$,
$W_4(\alpha_\lambda^+,\mu)$,
$W_2(\mu)$, $W_2(\mu)$,
$W_3(\lambda)$, and $W_3(\lambda)$,
respectively.  We then have the following identification between
the string algebras and the endomorphism spaces.

\begin{align*}
A_{2j,2k,2l}&=
\End((\bar\lambda\lambda)^j(\bar\iota\iota)^l(\mu\bar\mu)^k),\\
A_{2j,2k+1,2l}&=
\End((\bar\lambda\lambda)^j(\bar\iota\iota)^l(\mu\bar\mu)^k\mu),\\
A_{2j+1,2k,2l}&=
\End(\lambda(\bar\lambda\lambda)^j(\bar\iota\iota)^l(\mu\bar\mu)^k),\\
A_{2j+1,2k+1,2l}&=
\End(\lambda(\bar\lambda\lambda)^j(\bar\iota\iota)^l(\mu\bar\mu)^k\mu),\\
A_{2j,2k,2l+1}&=\End((\alpha_{\bar\lambda}^+\alpha_\lambda^+)^j
\iota(\bar\iota\iota)^l(\mu\bar\mu)^k),\\
A_{2j,2k+1,2l+1}&=\End((\alpha_{\bar\lambda}^+\alpha_\lambda^+)^j
\iota(\bar\iota\iota)^l(\mu\bar\mu)^k\mu),\\
A_{2j+1,2k,2l+1}&=\End(\alpha_\lambda^+(\alpha_{\bar\lambda}^+
\alpha_\lambda^+)^j\iota(\bar\iota\iota)^l(\mu\bar\mu)^k),\\
A_{2j+1,2k+1,2l+1}&=\End(\alpha_\lambda^+(\alpha_{\bar\lambda}^+
\alpha_\lambda^+)^j\iota(\bar\iota\iota)^l(\mu\bar\mu)^k\mu),\\
\end{align*}

Note that for inclusions such as $A_{2j,2k,2l}\subset A_{2j,2k,2l+1}$,
we use unitary equivalences 
$\alpha^+_\lambda\iota\cong\iota\lambda$,
$\alpha^+_{\bar\lambda}\iota\cong\iota\bar\lambda$,
$\lambda\bar\iota\cong\bar\iota\alpha_\lambda^+$, and
$\bar\lambda\bar\iota\cong\bar\iota\alpha_{\bar\lambda}^+$.
These are compatible with ways of identification of strings
with the connections.  (The only nontrivial identification
is done with $W_3(\lambda)$, where we use unitary equivalence
of $\iota\lambda$ and $\alpha_\lambda\iota$ as in Fig.\ref{conn3e}.)

By taking unions over $k$ and making the GNS-completions 
with respect to the compatible trace, we have a
commuting square of hyperfinite type II$_1$ factors as
in \cite[Assumption 1.1]{K2}.  We see that our triple
sequence of string algebras arise from this commuting
square as in \cite[Section 3]{K2}.
See \cite[Section 5]{K2} for
the case of the Dynkin diagrams of type $A$-$D$-$E$.


Our construction of $\alpha$-induced connections from $N\subset M$
uses information on all the $N$-$M$ morphisms and their intertwiners.
This is also true for Ocneanu's chiral generator
picture in Fig.~\ref{chigen} in the double triangle algebra.
This is theoretically fine, but we would like to have a method
to obtain the $\alpha$-induced connections purely in terms of
connections.  We discuss such a method at the end of this paper.

In the original setting of the set $\Delta$ of irreducible
endomorphisms of $N$, we fix $\mu\in\Delta$
and suppose we have a family of connections $W_1(\lambda,\mu)$
for $\lambda\in\Delta$ and the horizontal top and bottom
graphs for all of them are the same finite bipartite graph $\G$.
Let $*$ be the vertex of $\G$ corresponding to the identity
automorphism in $\Delta$.
Our positive braiding gives equivalence of the two composite connections
$W_1(\lambda_1,\mu)\cdot W_1(\lambda_2,\mu)$ and
$W_2(\lambda_1,\mu)\cdot W_1(\lambda_1,\mu)$.  This is given by
Fig.~\ref{braiding}, where $S$ and $T$ give unitary matrices
giving vertical gauge choices arising from a positive
braiding between $\lambda_1$ and $\lambda_2$.

\begin{figure}[H]
\begin{center}
\begin{tikzpicture}[scale=1]
\draw [thick, ->] (1.5,3)--(1,2);
\draw [thick, ->] (1.5,3)--(2,2);
\draw [thick, ->] (1,2)--(1.5,1);
\draw [thick, ->] (2,2)--(1.5,1);
\draw [thick, ->] (1.5,1)--(4.5,1);
\draw [thick, ->] (1.5,3)--(4.5,3);
\draw [thick, ->] (2,2)--(4,2);
\draw [thick, ->] (4.5,3)--(4,2);
\draw [thick, ->] (4.5,3)--(5,2);
\draw [thick, ->] (4,2)--(4.5,1);
\draw [thick, ->] (5,2)--(4.5,1);
\draw (3,2.5)node{$W_1(\lambda_1,\mu)$};
\draw (3,1.5)node{$W_1(\lambda_2,\mu)$};
\draw (1.5,2)node{$S$};
\draw (4.5,2)node{$T$};
\draw (5.5,2)node{$=$};
\draw [thick, ->] (6,3)--(6,2);
\draw [thick, ->] (6,2)--(6,1);
\draw [thick, ->] (8,3)--(8,2);
\draw [thick, ->] (8,2)--(8,1);
\draw [thick, ->] (6,1)--(8,1);
\draw [thick, ->] (6,2)--(8,2);
\draw [thick, ->] (6,3)--(8,3);
\draw (7,2.5)node{$W_1(\lambda_2,\mu)$};
\draw (7,1.5)node{$W_1(\lambda_1,\mu)$};
\end{tikzpicture}
\end{center}
\caption{Braiding $\e^+(\lambda_1,\lambda_2)$}
\label{braiding}
\end{figure}

Suppose we have a flat connection $W_0$ with respect to $*$
(in the sense of \cite[Definition 11.16]{EK2}) with the horizontal top
and bottom graphs being finite bipartite graphs $\G$ and $\H$
such that the composition $W_0\cdot\bar W_0$ decomposes into
a direct sum of irreducible connections each of which is
equivalent to one of $W_1(\lambda,\mu)$.  This is a connection
version of our setting in Section \ref{prelim} and
our connection $W_0$ must be of the form $W_2(\mu)$ for some $\iota$.

Construct a double sequence of string algebras $\{B_{kl}\}_{kl}$
as in \cite[Section 11.3]{EK2} from the connection $W_0$.  
That is, the commuting square
$\begin{array}{ccc}
B_{2l,2k}&\subset& B_{2l,2k+1}\\
\cap && \cap\\
B_{2l+1,2k}&\subset& B_{2l+1,2k+1}
\end{array}$ is described with $W_0$.  We further construct
a double sequence of string algebras $\{C_{kl}\}_{kl}$
so that 
$\begin{array}{ccc}
B_{2l,2k}&\subset& B_{2l,2k+1}\\
\cap && \cap\\
C_{2l,2k}&\subset& C_{2l,2k+1}
\end{array}$ is described with $W_1(\lambda,\mu)$ and
$\begin{array}{ccc}
C_{2l,2k}&\subset& C_{2l,2k+1}\\
\cap && \cap\\
C_{2l+1,2k}&\subset& C_{2l+1,2k+1}
\end{array}$ is described with $W_0$ .
Because of the braiding between $W_1(\lambda,\mu)$ and
$W_0\cdot\bar W_0$, we have compatibility between
\[
\begin{array}{ccc}
B_{2l,k}&\subset& B_{2l,k+1}\\
\cap && \cap\\
C_{2l,k}&\subset& C_{2l,k+1}\\
\cap && \cap\\
C_{2l+2,k}&\subset& C_{2l+2,k+1}
\end{array}
\]
and 
\[
\begin{array}{ccc}
B_{2l,k}&\subset& B_{2l,k+1}\\
\cap && \cap\\
B_{2l+2,k}&\subset& B_{2l+2,k+1}\\
\cap && \cap\\
C_{2l+2,k}&\subset& C_{2l+2,k+1}
\end{array}
\]
as in the case of the Intertwining Yang-Baxter equation.

Consider an element in $B_{2l+1,k}$.  It is embedded into
$B_{2l+2,k}$ and then into $C_{2l+2,k}$.  It is not a priori
clear whether the image
in $C_{2l+2,k}$ or not, but actually the image is in $C_{2l+1,k}$
because we are in the setting of the triple sequence of
string algebras introduced in this Section.  This gives
an embedding of $B_{2l+1,k}$ into $C_{2l+1,k}$.  Then the
results in Section \ref{alpha} mean that the connection
describing the commuting square
$\begin{array}{ccc}
B_{2l+1,2k}&\subset& B_{2l+1,2k+1}\\
\cap && \cap\\
C_{2l+1,2k}&\subset& C_{2l+1,2k+1}
\end{array}$ 
is $W_4(\alpha_\lambda^+,\mu)$.

\end{document}